\documentclass{amsart}
\usepackage{amsmath,amsfonts,amsthm,amssymb,indentfirst,epic,url,graphics,needspace}

\newtheorem{theorem}{Theorem}
\newtheorem{lemma}[theorem]{Lemma}

\newtheorem{proposition}[theorem]{Proposition}
\newtheorem{definition}[theorem]{Definition}
\newtheorem{question}[theorem]{Question}

\newtheorem{convention}[theorem]{Convention}

\newcommand{\Z}{\mathbb{Z}}
\newcommand{\N}{\mathbb{N}}
\newcommand{\Pm}{\mathcal{P}}
\renewcommand{\r}{\mathrm}
\newcommand{\supp}{\r{supp}}
\newcommand{\cl}{\r{cl}}
\newcommand{\po}{\rightsquigarrow}

\newcommand{\lang}{\begin{picture}(5,7)
\put(1.2,2.5){\rotatebox{45}{\line(1,0){6.0}}}
\put(1.2,2.5){\rotatebox{315}{\line(1,0){6.0}}}
\end{picture}\kern.16em}
\newcommand{\rang}{\kern.1em\begin{picture}(5,7)
\put(.1,2.5){\rotatebox{135}{\line(1,0){6.0}}}
\put(.1,2.5){\rotatebox{225}{\line(1,0){6.0}}}
\end{picture}}

\begin{document}

\title%
[Embeddability of rings in division rings]%
{Some results relevant to embeddability of rings\\
(especially group algebras) in division rings}
\thanks{Archived at \url{http://arXiv.org/abs/1812.06123}\,.
After publication, any updates, errata, related references,
etc., found will be recorded at
\url{http://math.berkeley.edu/~gbergman/papers/}\,.
}


\author{George M. Bergman}
\address{University of California\\
Berkeley, CA 94720-3840, USA}
\email{gbergman@math.berkeley.edu}

\begin{abstract}
P.\,M.\,Cohn showed in 1971 that given a ring $R,$ to describe,
up to isomorphism, a division ring $D$ generated by a homomorphic
image of $R$ is equivalent to specifying the set of square matrices
over $R$ which map to singular matrices over $D,$ and he determined
precisely the conditions that such a set of matrices must satisfy.
The present author later developed another version of this data,
in terms of closure operators on free $\!R\!$-modules.

In this note, we examine the latter concept further, and show how
an $\!R\!$-module $M$ satisfying certain conditions
can be made to induce such data.
In an appendix we make some observations on Cohn's
original construction,
and note how the data it uses can similarly be induced by
an appropriate sort of $\!R\!$-module.

Our motivation is the longstanding question of whether,
for $G$ a right-orderable group and $k$ a field, the group
algebra $kG$ must be embeddable in a division ring.
Our hope is that the right $\!kG\!$-module \mbox{$M=k((G))$}
might induce a closure operator of the required sort.
We re-prove a partial result in this direction due to N.\,I.\,Dubrovin,
note a plausible generalization thereof which would give the
desired embedding, and also sketch some
thoughts on other ways of approaching the problem.
\end{abstract}
\maketitle

\section{Background}\label{S.intro}
A.\,I.\,Mal'cev~\cite{Malcev} and,
independently, B.\,H.\,Neumann~\cite{BHN} showed that
if $G$ is a group given with a {\em \mbox{$\!2\!$-sided}-invariant}
ordering, that is, a total ordering $\leq$ such that for
all $e,\,f,\,g,\,h\in G,$
\begin{equation}\begin{minipage}[c]{35pc}\label{d.2-sided}
$f\leq g\ \implies\  ef\leq eg$\ \ and\ \ $fh\leq gh,$
\end{minipage}\end{equation}
and if, for $k$ a field, we let $k((G))$ denote the
set of formal $\!k\!$-linear combinations
$\sum_{g\in G}\,\alpha_g\,g$ of elements of $G$ whose support,
\begin{equation}\begin{minipage}[c]{35pc}\label{d.supp}
$\supp(\sum_{g\in G}\,\alpha_g\,g)\ =\ \{g\in G\mid \alpha_g\neq 0\},$
\end{minipage}\end{equation}
is well-ordered, then $k((G))$ can be made a ring
in a natural way; in fact, a division ring.
This division ring contains the group algebra $kG,$
as the subalgebra of elements with finite support.

Now suppose $G$ is merely given with a {\em right-invariant} ordering,
that is, a total ordering satisfying
\begin{equation}\begin{minipage}[c]{35pc}\label{d.rt-ord}
$f\leq g\ \implies\  fh\leq gh,$
\end{minipage}\end{equation}
and again let $k((G))$ be the set of
formal $\!k\!$-linear combinations of elements of $G$ whose supports
are well-ordered.
This time we cannot extend the ring structure of $kG$ to $k((G))$
in any evident way: if we
try to take the formal product $ab$ of elements $a,b\in k((G)),$
the one-sided invariance of the ordering is not enough to guarantee
that only finitely many occurrences of each $g\in G$
arise when we multiply $ab$ out; and even when that is true, for
instance, when $a$ is a member of $G,$ the support of
the resulting formal sum $ab$ may fail to be well-ordered.

However, by~\eqref{d.rt-ord} we can make $k((G))$ a
right $\!kG\!$-module; and this module has been shown to
have a property that is very encouraging with respect to the possibility
of embedding $kG$ in a division ring:
Dubrovin~\cite{Dubrovin} shows that
every nonzero element of $kG$ acts invertibly on $k((G)).$

But it is not clear how to go further: if we form
the ring of $\!k\!$-linear endomorphisms of $k((G))$ generated
by the actions of the elements of $kG$ and their inverses, there is
no evident way to prove invertibility of
all nonzero elements of this larger ring;
so we are not in a position to iterate the adjunction of inverses.
Indeed, the question of whether group rings of all right-orderable
groups are embeddable in division rings is listed in \cite{Kourovka}
as Problem~1.6, attributed to A.\,I.\,Mal'cev, and dating from the first
(1965) edition of that collection of open problems in group theory.
(The still more general question of whether group rings
of all torsion-free groups embed in division
rings also appears to be open [{\em ibid.}, Problems~1.3 and 1.5].)

P.\,M.\,Cohn \cite{PMC_1971}-\cite{FRR+} showed that a homomorphism
from a not necessarily
commutative ring $R$ into a division ring $D$ can be studied
in terms of the set of square matrices over $R$ that become
singular over $D.$
He showed that this set of matrices determines the structure of
the division subring of $D$ generated by the image of $R,$ and
gave criteria for a set of matrices to arise in this
way (recalled in~\S\ref{S.PMC_versions} below);
he named sets of matrices satisfying those criteria
``prime matrix ideals''.
Subsequently, the present author showed in \cite{sfd_fr_mtrd}
that the same data can be described in terms
of closure operators on free $\!R\!$-modules of finite rank
(details recalled in~\S\ref{S.mtrd} below).

Something I did not notice then is that a structure with {\em most} of
the properties defining Cohn's prime matrix ideals, or
my classes of closure operators, is determined by
any right or left $\!R\!$-module $M.$
In~\S\ref{S.M} we develop these observations for the closure
operator construction, and describe the additional properties
that $M$ must have for the closure operator so induced
to satisfy {\em all} the required conditions.

We then give, in \S\S\ref{S.bij}-\ref{S.bij_via_Higman},
a slightly modified proof of the result of Dubrovin cited above,
and in~\S\S\ref{S.further?}-\ref{S.either/or} look at a plausible
strengthening of that
result which would lead to the conclusion that $k((G))$ has the
module-theoretic properties needed to induce, by the results
of \S\ref{S.M}, an embedding in a division ring.
In \S\S\ref{S.G*}-\ref{S.variants} we discuss some
other ideas that might be of use in tackling this problem.

Finally, in an appendix, \S\ref{S.PMC_versions}, we look at Cohn's
concept of a prime matrix ideal.
We note a discrepancy between the definition of that concept
used in most of his works, and a weaker definition
given in~\cite{SF}, and sketch an apparent difficulty with his
reasoning about the latter version.
But we record an argument supplied by Peter Malcolmson,
which shows that adding a
small additional condition to the weaker definition renders
it equivalent to the other, and show that, so modified, it
allows us to obtain prime matrix ideals from certain
$\!R\!$-modules $M$ in a way parallel to our results on
closure operators.

Let me remark, regarding
the concepts of $\!2\!$-sided and $\!1\!$-sided
orderability of groups, that though the former seems ``intrinsically''
more natural, the latter has considerable ``extrinsic'' naturality:
A group is right orderable if and only if it can be embedded in the
group of order-automorphisms of a totally ordered set, written
as acting on the right
\cite[Proposition~29.5]{Darnel}.
Here ``only if'' is clear, using the group's action on itself.
To see ``if'' we need, for any totally ordered set $A,$ a
way of right-ordering $\r{Aut}(A).$
To get this, index $A$ as $\{a_i\mid i\in\kappa\}$ for some ordinal
$\kappa,$ and for $s\neq t\in\r{Aut}(A),$ let $s\leq t$ if and only if
for the {\em least} $i$ such that $s_i\neq t_i,$ we have $s_i<t_i.$
(In this argument we could, in fact, merely let the $a_i$ run over
an order-dense subset of $A.$
Since familiar totally ordered sets such as the real line tend to
have explicit {\em countable} order-dense subsets, this construction can
be performed for such groups without using the axiom of choice
to get the desired indexing.)

Still another fascinating
characterization of the one-sided orderable groups
is that they are those groups embeddable in lattice-ordered
groups (groups with a
{\em partial} ordering under which they are lattices, and
which is $\!2\!$-sided-invariant)
\cite[Corollary~29.8]{Darnel}.

\section{Conventions}\label{S.cnvntns}

Throughout this note, rings are associative
with $1,$ ring homomorphisms respect $1,$ and modules are unital.
If $M$ is a right $\!R\!$-module and $X$ a subset of $M,$
then $XR$ denotes the submodule of $M$ generated by $X,$
i.e., the set of finite sums $\sum x_i\,r_i$ with $x_i\in X,$
$r_i\in R.$
These include the empty sum, $0;$ hence if $X=\emptyset,$ then
$XR$ is the zero submodule.

\section{Closure structures on free modules}\label{S.mtrd}

We review here the result of~\cite{sfd_fr_mtrd}
relating homomorphisms of a ring $R$ into division rings with
certain closure operators on free $\!R\!$-modules of finite rank.
Recall

\begin{definition}\label{D.closure}
If $X$ is a set, then a {\em closure operator} on $X$ means a
map $\cl$ from subsets of $X$ to subsets of $X,$ such that
for all $S,\,T\subseteq X,$
\begin{equation}\begin{minipage}[c]{35pc}\label{d.cl_>>}
$S\ \subseteq\ T\ \implies\ \cl(S)\ \subseteq\ \cl(T).$
\end{minipage}\end{equation}
\begin{equation}\begin{minipage}[c]{35pc}\label{d.cl_>}
$\cl(S)\ \supseteq\ S.$
\end{minipage}\end{equation}
\begin{equation}\begin{minipage}[c]{35pc}\label{d.cl_idpt}
$\cl(\cl(S))\ =\ \cl(S).$
\end{minipage}\end{equation}

A closure operator $\cl$ will be called {\em finitary} if
for all $S\subseteq X,$
\begin{equation}\begin{minipage}[c]{35pc}\label{d.cl_fin}
$\cl(S)\ =\ \bigcup_{\,\mbox{\em\scriptsize finite}
\ S_0\subseteq S}\ \cl(S_0).$
\end{minipage}\end{equation}
\end{definition}

(The most common term for a closure operator
satisfying~\eqref{d.cl_fin} is ``algebraic'', because that condition
is frequent in algebraic contexts.
But ``finitary'' seems more to the point.)

Now suppose $R$ is a ring,
and $f:R\to D$ a homomorphism into a division ring.
For every $n\geq 0,$ let us
define a closure operator $\cl_{R^n}$ on $R^n$
by looking at the induced map $f:R^n\to D^n,$ and sending each
$S\subseteq R^n$ to the inverse image
in $R^n$ of the right span over $D$ of the image of $S$ in $D^n.$
In writing this formally, it will be convenient to use the
same letter $f$ that denotes our homomorphism $R\to D$ to
denote also the induced homomorphisms of right $\!R\!$-modules,
$R^n\to D^n,$ for all $n\geq 0.$
Then our definition says that
\begin{equation}\begin{minipage}[c]{35pc}\label{d.cl}
$\cl_{R^n}(S)\ =\ f^{-1}(f(S)D)$\hspace{1em}
for $S\subseteq R^n.$
\end{minipage}\end{equation}

It is not hard to verify that this construction satisfies
the following five conditions for all $m,n\geq 0.$
\Needspace{4\baselineskip}
\begin{equation}\begin{minipage}[c]{35pc}\label{d.cl_cl}
$\cl_{R^n}$ is a closure operator on
the underlying set of the right
$\!R\!$-module $R^n,$ whose closed subsets are $\!R\!$-submodules.
\end{minipage}\end{equation}
\begin{equation}\begin{minipage}[c]{35pc}\label{d.cl_proper}
For all $n>0,$ $\cl_{R^n}(\emptyset)$ is a proper submodule of $R^n.$
\end{minipage}\end{equation}
\begin{equation}\begin{minipage}[c]{35pc}\label{d.cl_homs}
For every homomorphism of right $\!R\!$-modules $h:R^m\to R^n$
and every $\!\cl_{R^n}\!$-closed submodule $A\subseteq R^n,$
the submodule $h^{-1}(A)\subseteq R^m$ is $\!\cl_{R^m}\!$-closed.
\end{minipage}\end{equation}
\begin{equation}\begin{minipage}[c]{35pc}\label{d.cl_exch}
The closure operator $\cl_{R^n}$ has the {\em exchange property},
namely, for $S\subseteq R^n$ and $t,u\in R^n,$
if $u\notin\cl_{R^n}(S)$ but $u\in\cl_{R^n}(S\cup\{t\}),$
then $t\in\cl_{R^n}(S\cup\{u\}).$
\end{minipage}\end{equation}
\begin{equation}\begin{minipage}[c]{35pc}\label{d.cl_finitary}
The closure operator $\cl_{R^n}$ is finitary.
\end{minipage}\end{equation}

In~\cite{sfd_fr_mtrd}, I named families of closure operations
$(\cl_{R^n})_{n\geq 0}$ satisfying~\eqref{d.cl_cl}-\eqref{d.cl_finitary}
``proper coherent matroidal structures on free $\!R\!$-modules''
(``matroid'' being the standard term for a set $X$ given with a
finitary closure operator $\r{cl}$ having
the exchange property of~\eqref{d.cl_exch}),
and it was shown that every such structure determines a
homomorphism $f$ of $R$ into a division ring $D$
which induces the given operators via~\eqref{d.cl}, and which is,
up to embeddings of division rings, the unique such homomorphism.
By~\eqref{d.cl},
the kernel of that homomorphism is $\cl_{R}(\emptyset).$

Condition~\eqref{d.cl_homs} above is stated in
terms of inverse images of closed subsets.
It is also equivalent (given~\eqref{d.cl_cl}) to
a statement about {\em closures of images} of subsets, namely:
\begin{equation}\begin{minipage}[c]{35pc}\label{d.cl_homs'}
For every homomorphism of right $\!R\!$-modules $h:R^m\to R^n$
$(m,n\geq 0)$ and every subset $S\subseteq R^m,$
the submodule $h(\cl_{R^m}(S))$ of $R^n$ is
contained in $\cl_{R^n}(h(S)).$
\end{minipage}\end{equation}

Indeed, consider an arbitrary subset $S\subseteq R^m$ and
an arbitrary closed subset $A\subseteq R^n.$
Then~\eqref{d.cl_homs} is equivalent to the statement that
for any such sets, if $S\subseteq h^{-1}(A)$ then
$\cl_{R^m}(S)\subseteq h^{-1}(A),$ while~\eqref{d.cl_homs'}
is equivalent to the statement that for any such sets,
if $h(S)\subseteq A$ then $h(\cl_{R^m}(S))\subseteq A.$
These statements are clearly equivalent,
so~\eqref{d.cl_homs} and~\eqref{d.cl_homs'} are equivalent.

We remark that matroid theorists often require the underlying
sets of matroids to be finite; for instance,
this is assumed by Welsh~\cite{Welsh}, and only in his final
chapter does he discuss ways the theory can be
extended to infinite structures.
But for most algebraic applications, including those of this note,
the restriction to finite sets would be unnatural, and the
appropriate version in the infinite case is clear:  Regarding
matroids as sets with closure operators
(one of many equivalent formulations of the concept),
one should simply require that these operators be finitary, i.e.,
one should impose condition~\eqref{d.cl_finitary}.
We shall call on many results from~\cite{Welsh} in this note,
tacitly understanding that the statements we quote go over to
the infinite matroids we will be considering.
The assumption that our closure operators are finitary makes it
straightforward to deduce such statements from the corresponding
facts about finite matroids.

(The term ``matroid'' is based on the
motivating example of the linear dependence structure on the
rows or columns of a matrix over a field $K.$
From that point of view, the finiteness assumption is natural.
But such systems of rows or columns are simply
finite families of elements of a space $K^n,$ and to the
algebraist, linear dependence
is most naturally viewed as structure on that generally infinite set.)

In the situations we shall be looking at,
conditions~\eqref{d.cl_cl}-\eqref{d.cl_homs} will generally
be easy to establish.
The next lemma restricts the instances one has to verify to show
that~\eqref{d.cl_exch} also holds, and shows
that~\eqref{d.cl_finitary} is
implied by \eqref{d.cl_cl} and~\eqref{d.cl_exch}.

\begin{lemma}\label{L.<n}
Let $R$ be a ring, $n$ a nonnegative integer, and $\cl_{R^n}$
an operator on subsets of $R^n$ satisfying~\eqref{d.cl_cl}.
Then\\[.2em]
\textup{(i)}~~$\r{cl}_{R^n}$ satisfies
the exchange property~\eqref{d.cl_exch} if and only if it satisfies
the restriction of that condition to sets
$S\subseteq R^n$ of cardinality~$<n,$ i.e., the condition
\begin{equation}\begin{minipage}[c]{35pc}\label{d.cl_exch-}
For every subset $S\subseteq R^n$ of cardinality $<n,$
and every pair of elements $t,u\in R^n,$
if $u\notin\cl_{R^n}(S)$ but $u\in\cl_{R^n}(S\cup\{t\}),$
then $t\in\cl_{R^n}(S\cup\{u\}).$\vspace{.2em}
\end{minipage}\end{equation}

Moreover, if $\r{cl}_{R^n}$ does satisfy~\eqref{d.cl_exch-},
then\\[.2em]
\textup{(ii)}~\,For all $S\subseteq R^n,$
there exists $S_0\subseteq S$ of cardinality $\leq n$
such that $\cl_{R^n}(S_0)=\cl_{R^n}(S),$ and\\[.2em]
\textup{(iii)}~For $S\subseteq R^n$ whose closure is
a proper subset of $R^n,$ there exists $S_0$
as in~\textup{(ii)} of cardinality $<n.$

Hence\\[.2em]
\textup{(iv)}~If an operator $\cl_{R^n}$ satisfies~\eqref{d.cl_cl}
and~\eqref{d.cl_exch}, it also satisfies~\eqref{d.cl_finitary}.
\end{lemma}

\begin{proof}
Let us first show that~\eqref{d.cl_exch-} implies~(ii) and~(iii).

Since $R^n$ is generated as an $\!R\!$-module by
the standard basis $e_1,\dots,e_n,$ and since by~\eqref{d.cl_cl},
closed subsets of $R^n$ are submodules,
we have $\cl_{R^n}(\{e_1,\dots,e_n\})=R^n.$

Now the conclusion of~(ii) that $\cl_{R^n}(S)$ is the
closure of a $\!{\leq}n\!$-element subset of $S$
is trivial if $S\subseteq \cl_{R^n}(\emptyset);$
in the contrary case, let $s_1\in S-\cl_{R^n}(\emptyset),$
and choose a subset $\{e_{i_1},\dots,e_{i_m}\}$
of $\{e_1,\dots,e_n\}$ minimal for having $s_1$ in its closure.
Thus by~\eqref{d.cl_exch-} with $\{e_{i_1},\dots,e_{i_{m-1}}\}$
in the role of $S,$ we have
$e_{i_m}\in\cl_{R^n}(\{e_{i_{m_1}},\dots,e_{i_{m-1}},s_1\});$
hence one can replace $e_{i_m}$ by $s_1$ in the
relation $\cl_{R^n}(\{e_1,\dots,e_n\})=R^n.$

If $S\not\subseteq\cl_{R^n}(\{s_1\})$
(which by the above observation can only happen if $n>1),$
then taking $s_2\in S-\cl_{R^n}(\{s_1\}),$
the corresponding argument shows that we can
replace another $e_i$ by $s_2;$ and so on.
Since there are only $n$ elements $e_i$ to be replaced,
this process must stop after $\leq n$ steps, giving a
subset of $\leq n$ elements of $S$ which (because the
process has stopped) has $S$ in its closure,
hence has the same closure as $S,$ proving~(ii).

Further, if $\cl_{R^n}(S)\neq R^n,$ this process can't
terminate with all the $e_i$ replaced by elements of $S,$
since that would imply that $S$ had
closure $R^n;$ so it must terminate
with $<n$ elements so replaced.
Again, the fact that the process has terminated means
that the set of $<n$ elements by which we have replaced
those elements has closure $\cl_{R^n}(S),$ establishing~(iii).

Now to get~(i), note that the exchange property of~\eqref{d.cl_exch}
for the closure operator
$\cl_{R^n}$ clearly implies~\eqref{d.cl_exch-}.
To get the converse, assume~\eqref{d.cl_exch-} and suppose we are
given $S\subseteq R^n,$ and $t$ and $u$
satisfying $u\notin\cl_{R^n}(S)$ but $u\in\cl_{R^n}(S\cup\{t\}).$
Thus $\cl_{R^n}(S)\neq R^n,$ so since~\eqref{d.cl_exch-}
implies~(iii), there exists $S_0\subseteq S$ of cardinality $<n$
such that $\cl_{R^n}(S_0)=\cl_{R^n}(S).$
Hence $u\notin\cl_{R^n}(S_0)$ and $u\in\cl_{R^n}(S_0\cup\{t\});$
so~\eqref{d.cl_exch-} gives
$t\in\cl_{R^n}(S_0\cup\{u\}),$ whence $t\in\cl_{R^n}(S\cup\{u\}),$
establishing~\eqref{d.cl_exch}.

Clearly~(ii) implies that $\cl_{R^n}$ is finitary, proving~(iv).
\end{proof}

Statements~(ii) and~(iii) above are instances of well-known properties
of matroids $(X,\cl)$ for which the whole set $X$ is the closure
of an $\!n\!$-element subset;
cf.~\cite[Corollary to Theorem~1.5.1, p.\,14]{Welsh}.
The method of proof of statement~(i) likewise yields a general result:
if $\cl$ is a closure operator on a set $X$
such that $X$ is the closure under $\cl$ of an $\!n\!$-element subset,
then $(X,\cl)$ is a matroid if and only if it satisfies the
weakened version of the exchange property in
which $S$ is restricted to subsets of cardinality $<n.$
I haven't seen this stated, but it is probably known.

Statement~(iv) is \cite[Lemma~2]{sfd_fr_mtrd}.
As noted in~\cite{sfd_fr_mtrd}, I included finitariness in my
list of conditions on the families of operators
considered so that these would clearly be matroid structures;
but that lemma showed the finitariness
condition superfluous in the presence of the other conditions.
So in the remainder of this note, the set of conditions
on a family of closure operators that we shall understand need to be
verified to get a homomorphism into a
division ring will be~\eqref{d.cl_cl}-\eqref{d.cl_exch}.

\section{Closure structures on restricted families of free modules}\label{S.n<N}

We shall consider in this section families
$(\cl_{R^n})_{0\leq n\leq N}$ of closure
operators $\cl_{R^n}$ defined only for the finitely many values of $n$
indicated in the subscript.
The hope is that results on such families
may prove useful in inductive proofs that certain
infinite families $(\cl_{R^n})_{n\geq 0}$
satisfy~\eqref{d.cl_cl}-\eqref{d.cl_exch} for all $n.$
The results of this section will not be called on in later sections,
so some readers may prefer to skip or skim this material.

\begin{convention}\label{C.N}
In this section, $R$ will be a ring and $N$ a fixed nonnegative integer,
and for $0\leq n\leq N,$ closure operators $\cl_{R^n}$ on $R^n$
will be assumed given, which satisfy~\eqref{d.cl_cl}-\eqref{d.cl_exch}
for all $m,n\leq N.$
\end{convention}

It will be useful to regard elements of $R^n$ as column vectors over
$R,$ and to treat finite families of such elements as matrices.
Let us fix some conventions regarding these.

Given an $n\times m$ matrix $H,$ we shall understand a
{\em submatrix} of $H$ to be
specified by a (possibly empty) subset of the $n$ row-indices and a
(possibly empty) subset of the $m$ column-indices.
(Thus, submatrices determined by different pairs of subsets will be
regarded as distinct, even if, when re-indexed using index-sets
$1,\dots,n_0$ and $1,\dots,m_0,$ they happen to give equal matrices.)
We shall regard the set of submatrices of $H$ as ordered by inclusion
(corresponding to inclusions among the sets of row-indices
and column-indices involved), so that we can speak of submatrices
maximal or minimal for a property.
A submatrix of $H$ will be called square if it has the same number
of rows as of columns, even if these are not indexed by the same
families of integers.
Because the statements we will be considering will not be
affected by rearranging the rows and columns, we shall, however,
for ease in visual presentation, often assume without loss of generality
that submatrices we are interested in form contiguous blocks.

For $R$ and $(\cl_{R^n})_{0\leq n\leq N}$ as in
Convention~\ref{C.N}, let us give names
to the properties of matrices over $R$ with $n\leq N$ rows which, in
the special case where $R$ is
a division ring and $\cl_{R^n}(S)$ denotes the
right subspace generated by $S,$
describe right, left, and $\!2\!$-sided invertibility of such matrices.

We remark that since conditions~\eqref{d.cl_cl}-\eqref{d.cl_exch}
are all stated in terms of {\em right} module structures, what we can
say about these matrix conditions will not, in general,
be symmetric with respect to rows and columns.
Note also that an index $m,$ $n$ etc.\ may or may not be
restricted to values $\leq N,$ depending on whether it occurs
in a context where it describes the heights of column vectors.

\begin{definition}\label{D.strong}
An $n\times m$ matrix $H$ over $R$ with $n\leq N$ will be called
{\em right strong} with respect to $\cl_{R^n}$ if the
closure under that operation of its set of columns is all of $R^n,$
{\em left strong} with respect to $\cl_{R^n}$
if no proper subfamily of its columns has closure equal to
the closure of all the columns,
and {\em strong} with respect to $\cl_{R^n}$
if both those conditions hold.
When the closure operator in question is clear from
context, we shall simply write ``right strong'',
``left strong'', and ``strong''.
\end{definition}

If operators $\cl_{R^n}$ are given for all $n\geq 0$
(and all cases of~\eqref{d.cl_cl}-\eqref{d.cl_exch} are thus assumed),
then the remaining results of this
section are easy to prove using the homomorphism
of $R$ into a division ring $D$ discussed
following~\eqref{d.cl_cl}-\eqref{d.cl_exch} above,
together with basic linear algebra over division rings.
But since here we are only assuming such operators given
for $n\leq N,$ we shall have to do things the hard way.

First, a few very basic facts,
though some of them are lengthy to establish.

\begin{lemma}\label{L.strong_basics}
\textup{(i)}~For all $n\leq N,$ the $n\times n$ identity
matrix $I_n$ is strong.

\textup{(ii)}\, Among matrices over $R$ with $\leq N$ rows, the class
of right-strong matrices and the class of left-strong matrices
are each closed under matrix multiplication
\textup{(}where defined\textup{)}.

\textup{(iii)} The class of right-strong matrices is also closed
under adjoining
additional columns and under deleting rows, and the class
of left-strong matrices under deleting
columns and adjoining rows \textup{(}so long as the number
of rows remains $\leq N).$
\end{lemma}

\begin{proof}
(i):  The assertion that $I_n$ is right strong says
that $\cl_{R^n}(\{e_1,\dots,e_n\})=R^n,$ which was noted
in the second sentence of the proof of Lemma~\ref{L.<n}.
The assertion that $I_n$ is left strong says that none
of these elements is in the closure of all the others.
To see this, note that for each $i,$
the inverse image of $\cl_R(\emptyset)$ under the $\!i\!$-th projection
map $R^n\to R$ is closed by~\eqref{d.cl_homs},
and contains $e_j$ for all $j\neq i,$ but by~\eqref{d.cl_proper}
does not contain $e_i.$

To get the first assertion of~(ii),
suppose $A$ and $B$ are right-strong matrices,
where $A$ is $n\times n'$ and $B$ is $n'\times n''$
with $n,\,n'\leq N.$
Let $a: R^{n'}\to R^n$ and
$b: R^{n''}\to R^{n'}$ be the linear maps on column vectors defined
by left multiplication by these matrices.
The image of $a$ is the right $\!R\!$-submodule of $R^n$
spanned by the columns of $A,$ so the assumption that $A$ is
right strong implies (and, given~\eqref{d.cl_cl},
is equivalent to saying) that the closure of that image is all of $R^n.$
Similarly, the assumption that $B$ is right strong
says that the closure of the image of $b$ is all of $R^{n'}.$
Now by~\eqref{d.cl_homs'}, the closure of the image of $ab$
contains the image under $a$ of the closure of the image of $b,$
in other words, $a(R^{n'}),$ hence it contains
$\cl_{R^n}(a(R^{n'})),$ which we have noted is $R^n,$ proving that
$AB$ is right strong.

The proof of the statement about left-strong matrices is longer,
and is most easily carried out with the help of some concepts
and results from the theory of matroids.
Let us call a family $(x_i)_{i\in I}$ of elements of
$R^m$ $(m\leq N)$ {\em independent}
if the closure of $\{x_i\mid i\in I\},$ which for brevity we will
call the closure of the $\!I\!$-tuple $(x_i)_{i\in I},$
is not the closure of any proper subfamily
$(x_i)_{i\in J}$ $(J\subsetneqq I);$
and let the {\em rank} of an arbitrary family $(x_i)_{i\in I}$
mean the cardinality of any independent subfamily
$(x_i)_{i\in J}$ $(J\subseteq I)$ having
the same closure as the whole family, that is, any subfamily
minimal for having that closure.
Such subfamilies exist by Lemma~\ref{L.<n}(ii),
and by a standard result on matroids, based on an
element-by-element replacement construction as in the proof
of Lemma~\ref{L.<n}(ii), their cardinalities are all the same
value $\leq m$ \cite[Corollary to Theorem~1.5.1, p.14]{Welsh},
justifying the above definitions.
That same replacement argument shows that every independent family
of elements of $R^m$ can be extended to an $\!m\!$-element independent
family whose closure is all of $R^m$ \cite[Theorem~1.5.1]{Welsh}.

Now suppose $A$ is an $n\times n'$ left strong matrix
(with $n\leq N,$ and hence $n'$ necessarily also $\leq N$
by the left-strong condition),
let $a: R^{n'}\to R^n$ be the map it determines,
and let $(x_i)_{i\in I}$
be any independent family of elements of $R^{n'}.$
By the above observations, this can be extended to an
$\!n'\!$-element independent family
$(x_i)_{i\in J}$ $(J\supseteq I,$ $\r{card}(J)=n')$
with closure all of $R^{n'}.$
Hence by~\eqref{d.cl_homs'},
\begin{equation}\begin{minipage}[c]{35pc}\label{d.xiei}
$\cl_{R^n}\{a(x_i)\mid i\in J\}
=\cl_{R^n}(a(\cl_{R^{n'}}\{x_i\mid i\in J\}))
=\cl_{R^n}(a(R^{n'}))
=\cl_{R^n}\{a(e_i)\mid 1{\leq} i {\leq}n'\},$
\end{minipage}\end{equation}
where $\{e_1,\dots,e_{n'}\}$ is the standard basis of $R^{n'}.$
Now $\{a(e_i)\mid 1\leq i\leq n'\}$ is independent because $A$ is
left strong.
If $(a(x_i))_{i\in J}$ were dependent, then its closure
would be the closure of an independent family of $<n'$ elements,
so~\eqref{d.xiei} would
contradict the equal-cardinalities result cited in the
next-to-last sentence of the preceding paragraph.
So $(a(x_i))_{i\in J}$ is independent,
so a fortiori, $(a(x_i))_{i\in I}$ is independent.
Thus, $a$ carries independent families to independent families.

But this is equivalent to the result we are trying to prove; for if
$B$ is an $n'\times n''$ left strong matrix
and $b: R^{n''}\to R^{n'}$ its action, then
the statement that $B$ is left strong says that the family
of its columns, $(b(e_i))_{1\leq i\leq n''},$ is independent, and
the above result shows the same for $(ab(e_i))_{1\leq i\leq n''},$
the family of columns of~$AB.$
This completes the proof of~(ii).

In~(iii), the assertion
that the class of right-strong matrices is closed under adjoining
additional columns is clear from the definition of right strong.

That left strong matrices are closed under adjoining rows is
most easily verified as the contrapositive statement, that the class of
non-left-strong matrices is closed under dropping rows.
This is equivalent to saying that a dependent family of column
vectors remains dependent on dropping from each of its
members a fixed subset of the entries.
Such an entry-dropping operation corresponds to a homomorphism
$R^m\to R^n,$ and from~\eqref{d.cl_homs'} it is easy to see
that the image of a dependent family under {\em any} homomorphism
is dependent, giving the desired result.

In showing that dropping rows preserves the property
of being right strong, and dropping columns that
of being left strong, let our matrix be $n\times n',$
and assume without loss of
generality that the rows or columns to be dropped are
the last $d$ rows or columns, for some $d.$
Then the row-dropping operation is equivalent to left multiplication by
the $n{-}d\times n$ matrix $(I_{n-d}\ \ 0),$
where $I_{n-d}$ is the $n{-}d\times n{-}d$ identity matrix, and
the column-dropping operation to right multiplication by the
$n'\times n'{-}d$ matrix
$\left(\begin{matrix} I_{n'-d}\\
0\end{matrix}\right).$
That these matrices are respectively right and left strong
can be seen from the argument proving~(i);
the desired results then reduce to special cases of~(ii).
\end{proof}

In some considerations, it is easy to deal with adding and deleting
columns of a matrix, because these operations keep us within the
same set $R^n,$ but harder to handle adding and deleting rows.
In such situations, the following observation will be helpful.

\begin{lemma}\label{L.e_i}
Let $H$ be an $n\times n'$ matrix over $R,$ with $0\leq n\leq N,$
and let us be given a list of indices $0<i_1<\dots<i_d\leq n.$
Let $H'$ be the $n{-}d\times n'$ matrix formed by deleting
from $H$ the rows with indices $i_1,\dots,i_d,$ and
$H''$ the $n\times n'{+}d$ matrix formed by adjoining to $H$
the additional columns $e_{i_1},\dots,e_{i_d}.$

Then $H'$ is right strong, respectively, left strong, if and
only if $H''$ is.
\end{lemma}

\begin{proof}[Sketch of proof]
Let $f: R^n\to R^{n-d}$ be the map which drops the coordinates
with indices $i_1,\dots,i_d,$ and $g: R^{n-d}\to R^n$ the right
inverse to $f$ which inserts $\!0\!$'s in these positions.
We see that $f^{-1}$ and $g^{-1},$ applied to
submodules, give a bijection between
all submodules of $R^{n-d},$ and those submodules
of $R^n$ which contain $e_{i_1},\dots,e_{i_d}.$
By~\eqref{d.cl_homs}, $f^{-1}$ and $g^{-1}$ each carry closed
submodules to closed submodules, hence by the preceding
observation they induce a
bijection between $\!(\cl_{R^{n-d}})\!$-closed submodules of $R^{n-d},$
and $\!\cl_{R^n}\!$-closed submodules
of $R^n$ that contain $e_{i_1},\dots,e_{i_d}.$
It follows that if $S$ is a subset and $t$ an element of $R^n,$
then $t\in\cl_{R^n}(S\cup\{e_{i_1},\dots,e_{i_d}\})$ if and only if
$f(t)\in\cl_{R^{n-d}}(f(S)).$
The two conclusions of
the lemma follow by combining these
observations with the definitions of right and left strong.
\end{proof}

We can now generalize several facts about right and left invertible
matrices over a division ring to right and left strong
matrices with respect to the system $(\cl_{R^n})_{0\leq n\leq N}$
of Convention~\ref{C.N}.

\begin{lemma}\label{L.strong_n-1}
Let $H$ be an $n\times n'$ matrix over $R,$ with $n\leq N.$
Then

\textup{(i)} If $H$ is right strong, then $n\leq n',$
i.e., $H$ has at least as many columns as rows.

\textup{(ii)} If $H$ is left strong, then $n\geq n',$
i.e., $H$ has at least as many rows as columns.

\textup{(iii)} Suppose $H$ is right strong, and $H'$ is a
submatrix of $H$ given by a subset of the columns of $H.$
Then $H'$ is minimal among such submatrices
which are right strong if and only if it is maximal among such
submatrices which are left strong.
Hence if $H'$ has either the above minimality property or
the above maximality property, it is strong.

\textup{(iv)} Suppose $H$ is left strong, and $H'$ is a
submatrix of $H$ given by a subset of the rows of $H.$
Then $H'$ is minimal among such submatrices
which are left strong if and only if it is maximal among such
submatrices which are right strong.
Hence, again, if $H'$ has either the above minimality property or
the above maximality property, it is strong.

\textup{(v)} If $H$ is left or right strong,
then it is strong if and only if $n=n',$
i.e., if and only if it is square.

\textup{(vi)} If $m$ is the common cardinality of all
maximal independent sets of columns of $H,$
equivalently, all sets of columns minimal
for having the same closure as the set of all columns of $H,$
then all maximal strong submatrices of $H$ are $m\times m.$

\textup{(vii)} If $H$ is square, hence $n\times n,$
and its upper left-hand $n{-}1\times n{-}1$
block is strong, then $H$ is strong if and only if
its $\!n\!$-th column is not in the
closure of its other $n{-}1$ columns.
\end{lemma}

\begin{proof}
(i)~says that any
family of elements of $R^n$ whose closure is $R^n$ must have
$\geq n$ elements, and~(ii) that
any independent family must have $\leq n$ elements.
Both of these facts follow easily from the results cited in
the second paragraph of the proof of Lemma~\ref{L.strong_basics}(ii).

Assertion~(iii) is immediate:
With the help of the exchange property, it is
easy to check that the maximal independent subfamilies of the family
of columns of $H$ are the same as the minimal subfamilies having $R^n$
as closure (an instance of a general property of matroids).

To get~(iv), let $H^*$ be the matrix $(H\ I_n)$ obtained by
appending the columns $e_1,\dots,e_n$ to $H.$
Clearly, $H^*$ is right strong.
Let us now consider submatrices of $H^*$ which consist of all the
columns of $H,$ together with a subset of the columns of $I_n,$
and apply~(iii) to such submatrices.
It is easy to see that a submatrix given by such
a set of columns, which is minimal
or maximal among such submatrices for one of the properties
named in~(iii), will in fact be minimal or maximal for the same
property among {\em all} submatrices given by subsets
of the columns of $H^*.$
Applying~(iii), and then Lemma~\ref{L.e_i}, we get~(iv).

In~(v), the ``only if'' direction follows from~(i) and~(ii).
To get ``if'', suppose that $H$ is square and right strong.
Statement~(iii) tells us that by deleting some columns of $H$
we can get a strong matrix $H';$ but if the set of columns so
deleted were nonempty, then $H'$ would not be square, contradicting
the ``only if'' direction.
So that set of columns is empty, so $H=H'$ is itself strong.
In the case where $H$ is left strong, one uses the same
method, calling on~(iv) in place of~(iii).

The conclusion of~(vi) is equivalent to saying that every
strong submatrix of $H$ is contained in an $m\times m$
strong submatrix; so let $A$ be any strong submatrix of $H,$ and
assume without loss of generality \textup{(}by permuting
the rows and columns of $H$ if necessary\textup{)}
that $A$ is the upper left-hand $m'\times m'$ block of $H$
for some $m'\leq n.$
Since $A$ is strong it is left strong, hence by
Lemma~\ref{L.strong_basics}(iii), the submatrix $H'$ of $H$
given by the first $m'$ columns of $H$ is also left strong.
Let us extend $H'$ to a maximal left-strong submatrix
$H''$ of $H$ given by a subset of the columns; this will
have $m$ columns, which we can assume by a further rearrangement
are the first $m$ columns of $H.$
Now the submatrix given by the first $m'$ rows of $H''$
contains the columns of $A,$ hence is right strong; let us
adjoin further rows of $H''$ to get a maximal right strong
submatrix $A''$ consisting of rows of $H''.$
By~(iv) (with $H''$ in the role of the $H$ of that statement),
$A''$ is strong, so it is the desired $m\times m$ strong
submatrix of $H$ containing $A.$

Finally, in the situation of~(vii), note that the $n\times n{-}1$
submatrix given by the first $n-1$ columns of $H$ is left strong,
since it is obtained by adjoining a row to a left strong
$n{-}1\times n{-}1$ matrix.
The matrix $H$ is obtained from that submatrix
by bringing in a final column,
so it will be left strong if and only if that column is not
in the closure of the other columns.
(The verification of the ``if'' direction uses the exchange property.)
Since it is square,~(v) tells us that
that necessary and sufficient condition for it to
be left strong is in fact necessary and sufficient for it to be strong.
\end{proof}

The above arguments were, as noted, row-column asymmetric because
we were considering closure operators on column vectors,
rather than on row vectors.
Note, however, that the set of $m\times n$ matrices $(m,n\leq N)$
right strong with respect to $(\cl_{R^n})_{0\leq n\leq N}$
is easily seen to determine $(\cl_{R^n})_{0\leq n\leq N}.$
It seems likely that one could find simple necessary and sufficient
conditions for a set of matrices to be the right strong matrices
with respect to such a family.
If it should turn out that the left strong matrices
satisfy the analogs of those conditions, it would follow
that our family of closure operators on right modules is equivalent to
a similar family on left modules.
I have not investigated these ideas.
(If one developed such results, one would want distinct notations for
the sets of $\!n\!$-element rows and $\!n\!$-element columns over $R.$
Following Cohn, one might denote these $R^n$ and $^n\!R.)$

Let us also note that the existence of just finitely many
operators, $\cl_{R^n}$ for $0\leq n\leq N,$ as
discussed in this section, cannot alone be sufficient
for $R$ to be embeddable in a division ring.
For if $R$ is any {\em $\!N\!$-fir}, the method of
\cite[\S7]{sfd_fr_mtrd} yields such closure operators
for $0\leq n\leq N,$ but examples are known of $\!N\!$-firs
whose free modules of larger finite ranks behave badly,
e.g., satisfy $R^{N+1}\cong R^{N+2},$ making it
impossible to embed such rings in division rings.
Thus, if the results of this section are to be useful in proving
results on embeddability in division rings, that use
is likely to be, as suggested at the start of this section,
as a tool in inductive developments.
(Some details on the facts cited above:
The concept of {\em fir}, i.e., free ideal ring,
is recalled in \cite[\S7]{sfd_fr_mtrd}, and $R$ is assumed
to be a fir in most of the that section; but the
final paragraphs of that section note how to generalize the
arguments to the case of a {\em semifir}, and the same method,
applied to {\em $\!N\!$-firs}, rings in which every right or left
ideal generated by $\leq N$ elements is free of unique rank,
gives the desired family $(\cl_{R^n})_{0\leq n\leq N}.$
For examples of $\!N\!$-firs whose free modules of ranks ${>}\,N$ behave
badly, see the $V_{m,n}$ case of \cite[Theorem 6.1]{coproducts2}.)

\section{Systems of closure operators induced by $\!R\!$-modules}\label{S.M}
In \S\ref{S.mtrd} we saw how a homomorphism of a ring $R$
into a division ring $D$ induces, by~\eqref{d.cl},
a system of closure operators
satisfying~\eqref{d.cl_cl}-\eqref{d.cl_exch}.
Suppose that instead of
a homomorphism from $R$ to a division ring, we are given
a nonzero right $\!R\!$-module $M.$
There is no obvious way to put $M$ in place of $D$ in~\eqref{d.cl}
(even if we assume it a left rather than a right module);
but we shall see below that there is a natural way to get from $M$
a system of closure operators $(\cl_{R^n})_{0\leq n}$
which, for $M=D_R$ $(D$ regarded as a right $\!R\!$-module)
agrees with that given by~\eqref{d.cl}.

For each $n>0,$ let us write elements $a\in M^n$ as
row vectors, and elements $x\in R^n$ as column vectors.
Then for such $a$ and $x$ we can define $a x\in M$ in the obvious way;
thus we can speak of elements of $R^n$
annihilating elements of $M^n.$
For $S\subseteq R^n,$ let $\cl_{R^n}(S)$
be the set of elements of $R^n$ that annihilate all
elements of $M^n$ annihilated by all elements of $S.$
Writing $\r{ann}_{M^n}(S)$ for
$\{a\in M^n\mid (\forall\,s\in S)\linebreak[1]\ as=0\},$
this becomes
\begin{equation}\begin{minipage}[c]{35pc}\label{d.cl_fr_M}
$\cl_{R^n}(S)\ =\ \{x\in R^n\mid
\r{ann}_{M^n}(x)\supseteq \r{ann}_{M^n}(S)\}.$
\end{minipage}\end{equation}
We see that the closed subsets of $R^n$ under~\eqref{d.cl_fr_M}
are precisely the annihilators of subsets of $M^n.$

It is not hard to check that given a homomorphism to a division
ring, $f:R\to D,$ as in \S\ref{S.mtrd},
if we let $M=D_R,$ then~\eqref{d.cl_fr_M} describes the
same closure operator as~\eqref{d.cl}.
(The key observation is that every subspace of the right $\!D\!$-vector
space of height-$\!n\!$ columns over $D$ is the right annihilator
of a set of length-$\!n\!$ rows over $D$ -- since such row
vectors correspond to the $\!D\!$-linear functionals on that space --
so the right $\!R\!$-submodules of $R^n$ that are inverse
images under $f$ of $\!D\!$-subspaces of $D^n,$
regarded as sets of columns, are those that are
annihilators of sets of elements of~$D^n$ regarded as rows.)

Returning to the case of a general right $\!R\!$-module $M,$
let us, for any {\em matrix} $A$ over $R$ with $n$ rows,
write $\r{ann}_{M^n}(A)$ for the subset of $M^n$ annihilated
by the right action of $A,$ in other words,
the annihilator in $M^n$ of the set of columns of~$A.$

\begin{lemma}\label{L.cl_fr_M}
Let $R$ be a ring and $M$ a nonzero right $\!R\!$-module, and
for each $n\geq 0$ let $\cl_{R^n}$ be defined by~\eqref{d.cl_fr_M}.
Then this family of operators satisfies
conditions~\eqref{d.cl_cl}, \eqref{d.cl_proper} and
\eqref{d.cl_homs}.
For each $n,$ the condition that $\cl_{R^n}$ also
satisfy~\eqref{d.cl_exch}
\textup{(}which, as noted, is equivalent
to~\eqref{d.cl_exch-}\textup{)}, is equivalent
to each of the following three statements.
\begin{equation}\begin{minipage}[c]{35pc}\label{d.cl_exch_iff}
There do not exist a subset $S\subseteq R^n,$ and elements
$u,t\in R^n,$ such that\\
$\r{ann}_{M^n}(S)\ \supsetneqq\ \r{ann}_{M^n}(S\cup\{u\})\ \supsetneqq
\ \r{ann}_{M^n}(S\cup\{t\}).$
\end{minipage}\end{equation}
\begin{equation}\begin{minipage}[c]{35pc}\label{d.cl_exch_iff_mx-}
There do not exist an $n\times n{-}1$ matrix $A$ over $R,$
and $n\times n$ matrices $B,$ $C$ over $R,$
each obtained by adding a single column to $A,$ such that
\ $\r{ann}_{M^n}(A)\ \supsetneqq
\ \r{ann}_{M^n}(B)\ \supsetneqq\ \r{ann}_{M^n}(C).$
\end{minipage}\end{equation}
\begin{equation}\begin{minipage}[c]{35pc}\label{d.cl_exch_iff_mx}
There do not exist $n\times n$ matrices $A,$ $B,$ $C$
over $R$ which all agree except in one column, such that
\ $\r{ann}_{M^n}(A)\ \supsetneqq
\ \r{ann}_{M^n}(B)\ \supsetneqq\ \r{ann}_{M^n}(C).$
\end{minipage}\end{equation}

Hence, if a ring $R$ has a {\em faithful} right module $M$ which
for all $n\geq 0$ satisfies~\eqref{d.cl_exch_iff}, equivalently,
\eqref{d.cl_exch_iff_mx-}, equivalently, \eqref{d.cl_exch_iff_mx},
then $R$ is embeddable in a division ring.
\end{lemma}

\begin{proof}
That the operators defined by~\eqref{d.cl_fr_M} satisfy
\eqref{d.cl_cl} and \eqref{d.cl_proper} is straightforward,
the key fact being that the annihilator in $R^n$ of
every element $a\in M^n$ is a right submodule of $R^n,$ which
is proper if $a\neq 0.$

\eqref{d.cl_homs} is also not difficult, but here are the details.
Let $h:R^m\to R^n$ be represented by the
$n\times m$ matrix $H,$ acting on the left on columns
of elements of $R.$
The matrix $H$ can also be applied on the right to rows of elements of
$M,$ so as to carry $M^n$ to $M^m,$ and if we also call this
map $h$ (and write it on the right), the associativity
of formal matrix multiplication gives the law $(a\,h)\,x=a\,(h\,x).$
Thus, if $A\subseteq R^n$ is closed, i.e., is the
annihilator of a subset $T\subseteq M^n,$ and we write its
inverse image $h^{-1}(A)\subseteq R^m$ as
$\{x\in R^m\mid h\,x\in A\}=
\{x\in R^m\mid(\forall\,t\in\nolinebreak T)\ t\,(h\,x)=0\}=
\{x\in R^m\mid(\forall\,t\in T)\ (t\,h)\,x=0\},$ we see that
this is the annihilator of $T\,h\subseteq M^m,$ hence also closed.

The equivalence of~\eqref{d.cl_exch} with~\eqref{d.cl_exch_iff}
is easy to see if we bear in mind that an inclusion between the
annihilators in $M^n$ of two subsets of $R^n$ is
equivalent to the {\em reverse} inclusion between the closures
of those subsets of $R^n,$ as defined by~\eqref{d.cl_fr_M}.

Condition~\eqref{d.cl_exch_iff_mx-} is a translation
of~\eqref{d.cl_exch-}, gotten by
looking at the $<n$ elements of $R^n$ in~\eqref{d.cl_exch-}
as columns of a matrix (and if there are fewer than $n-1$
elements in the set, throwing in enough zero columns
to bring the number of columns up to $n-1),$
then applying the definition~\eqref{d.cl_fr_M}.

The difference between~\eqref{d.cl_exch_iff_mx-}
and~\eqref{d.cl_exch_iff_mx} is merely cosmetic.
Indeed, if three $n\times n$ matrices with the properties
referred to in~\eqref{d.cl_exch_iff_mx} exist, then letting
$A'$ be the $n\times n{-}1$ matrix gotten by deleting the
column in which those three differ, we get
$\r{ann}_{M^n}(A')\supseteq \r{ann}_{M^n}(A)\supsetneqq
\r{ann}_{M^n}(B)\supsetneqq\r{ann}_{M^n}(C),$ so $A',$ $B,$ $C$
have the properties referred to
in~\eqref{d.cl_exch_iff_mx-}, while conversely, given
$A,$ $B,$ $C$ as in~\eqref{d.cl_exch_iff_mx-}, if we expand $A$
to an $n\times n$ matrix by adjoining a zero column, we get
matrices with the properties of~\eqref{d.cl_exch_iff_mx}.

To see the final assertion of the lemma, note that
applying~\eqref{d.cl_fr_M} with $n=1$ and $S=\emptyset,$ we find that
$\cl_{R}(\emptyset)\subseteq R$ is the annihilator of $M.$
If $M$ is faithful, this is the trivial ideal of $R,$ so
by \cite[(21) and Theorem~22]{sfd_fr_mtrd},
the system of closure operators $\cl_{R^n}$
determines a homomorphism with zero kernel from $R$ to a division ring.
\end{proof}

(So if, for every right ordered group $G,$
we could prove that the right $\!kG\!$-module $k((G))$
satisfied~\eqref{d.cl_exch_iff}, equivalently~\eqref{d.cl_exch_iff_mx-},
equivalently~\eqref{d.cl_exch_iff_mx}, this would show $kG$ embeddable
in a division ring.)\vspace{.4em}

Above, we have obtained a closure operator on free
$\!R\!$-modules by comparing
{\em kernels} of the maps $M^n\to M$ induced by elements of $R^n.$
Can we get a similar construction using images rather than kernels?

Yes, but things have to be set up a bit differently.
Note that in the key case where $M=D_R$ for
$D$ a division ring, there are only two possibilities for the image
of a map $M^n\to M$ induced by an element of $R^n;$ so to get
useful structure, we should
instead look at images of maps $M\to M^n$ induced by such elements.
If we treat elements of $R^n$ as row vectors,
acting on a right $\!R\!$-module $M,$ we find that
the set of elements of $R^n$ determining maps whose images lie in
(say) the image of a given such map will not, in general, be a
right $\!R\!$-submodule of $R^n,$ but will be a left $\!R\!$-submodule.
We could ask when the resulting left-$\!R\!$-submodule-valued
closure operators satisfy the left-right duals of
conditions~\eqref{d.cl_cl}-\eqref{d.cl_exch}, which, by
symmetry, would also lead to homomorphisms into division rings;
but let us, instead, stay in the context
of~\eqref{d.cl_cl}-\eqref{d.cl_exch}
by starting with a {\em left} $\!R\!$-module $L,$
and carrying out the left-right dual of the construction just sketched.
Thus, given $L,$ we again regard elements
of $R^n$ as column vectors, but now let them act on the
left on $L,$ mapping elements of $L$ to column vectors over $L.$
We now define closure operators $\cl_{R^n}$ $(n\geq 0)$
by specifying that for $S\subseteq R^n,$
\begin{equation}\begin{minipage}[c]{35pc}\label{d.cl_fr_L}
$\cl_{R^n}(S)\ =\ \{x\in R^n\mid
x\,L\subseteq \sum_{s\in S} s\,L\}.$
\end{minipage}\end{equation}

We easily obtain the analog of Lemma~\ref{L.cl_fr_M}:

\begin{lemma}\label{L.cl_fr_L}
Let $R$ be a ring and $L$ a nonzero left $\!R\!$-module, and
for each $n\geq 0$ let $\cl_{R^n}$ be defined by~\eqref{d.cl_fr_L}.
Then this family of operators satisfies
conditions~\eqref{d.cl_cl}, \eqref{d.cl_proper} and \eqref{d.cl_homs}.
For each $n,$ the condition that $\cl_{R^n}$ also
satisfy the exchange property of~\eqref{d.cl_exch}
\textup{(}which, as noted, is equivalent
to~\eqref{d.cl_exch-}\textup{)} is equivalent
to each of the following three statements.
\begin{equation}\begin{minipage}[c]{35pc}\label{d.cl_exch_iff_L}
There do not exist a subset $S\subseteq R^n$ and elements
$u,t\in R^n$ such that\\
$\sum_{s\in S}s\,L\ \subsetneqq
\ \sum_{s\in S\cup\{u\}}s\,L\ \subsetneqq
\ \sum_{s\in S\cup\{t\}}s\,L.$
\end{minipage}\end{equation}
\begin{equation}\begin{minipage}[c]{35pc}\label{d.cl_exch_iff_mx-_L}
There do not exist an $n\times n{-}1$ matrix $A$
and $n\times n$ matrices $B$ and $C,$ each obtained by adding a single
column to $A,$ such that
\ $A\,L^{n-1}\ \subsetneqq\ B\,L^n\ \subsetneqq\ C\,L^n.$
\end{minipage}\end{equation}
\begin{equation}\begin{minipage}[c]{35pc}\label{d.cl_exch_iff_mx_L}
There do not exist $n\times n$ matrices $A,$ $B,$ $C$
over $R$ which all agree except in one column, such that
\ $A\,L^n\ \subsetneqq\ B\,L^n\ \subsetneqq\ C\,L^n.$
\end{minipage}\end{equation}

Hence, if a ring $R$ has a faithful left module $L$ which
for all $n\geq 0$ satisfies~\eqref{d.cl_exch_iff_L}, equivalently,
\eqref{d.cl_exch_iff_mx-_L}, equivalently, \eqref{d.cl_exch_iff_mx_L},
then $R$ is embeddable in a division ring.
\end{lemma}

\begin{proof}[Sketch of proof]
Again, the verifications of~\eqref{d.cl_cl}, \eqref{d.cl_proper}
and~\eqref{d.cl_homs} are straightforward, with that
of~\eqref{d.cl_homs} using associativity of matrix multiplication.
The proofs that~\eqref{d.cl_exch_iff_L},
\eqref{d.cl_exch_iff_mx-_L}, and~\eqref{d.cl_exch_iff_mx_L}
are all equivalent to~\eqref{d.cl_exch} parallel
the proofs for~\eqref{d.cl_exch_iff},
\eqref{d.cl_exch_iff_mx-}, and~\eqref{d.cl_exch_iff_mx}.
\end{proof}

(Let us note examples showing that in the situations of the
above two lemmas, condition~\eqref{d.cl_finitary} need not
hold if~\eqref{d.cl_exch} does not.
Let $k$ be a field, let $R\subseteq k^\N$ be the subring of all
{\em eventually constant} sequences of elements of $k,$
let $M=R,$ and let $L=S=$ the ideal of eventually-zero sequences.
It is not hard to verify that under the closure operator
$\cl_R,$ defined either as in Lemma~\ref{L.cl_fr_M} using $M$ or as in
Lemma~\ref{L.cl_fr_L} using $L,$ $\cl_R(S)=R,$ while the closure of any
finite subset of $S$ is the ideal that it generates, hence does
not contain $1\in R;$ so
$\cl(S)\neq\bigcup_{\,\mbox{\scriptsize finite}
\,S_0\subseteq S}\,\cl(S_0).)$
\vspace{.4em}

We remark that the description of matroids in terms of closure
operators is only one of many
surprisingly diverse, though ultimately equivalent, ways of
developing that concept \cite[Chapter~1]{Welsh}.
Moreover, matroid structures on modules $R^n,$ discussed above,
and prime matrix ideals, considered in \S\ref{S.PMC_versions} below,
are just two of several ways of describing
the data that determine a homomorphism from $R$ into a division ring;
for others, see~\cite{Malcolmson2}.

\section{The idea of Dubrovin's result}\label{S.bij}

Let us change gears, and in this and the next two
sections develop the result:
\begin{equation}\begin{minipage}[c]{35pc}\label{d.Dubrovin}
(After N.\,I.\,Dubrovin \cite{Dubrovin}.)
For $G$ a right ordered group, the right action of every nonzero
element of $kG$ on $k((G))$ is invertible.
\end{minipage}\end{equation}

In this section we sketch what is involved; in \S\ref{S.wo} we look
at an order-theoretic tool that can ``organize'' the proof,
and in \S\ref{S.bij_via_Higman}, we apply that tool
to recover~\eqref{d.Dubrovin}.
In \S\S\ref{S.further?}-\ref{S.either/or} we shall return to the
ideas of \S\ref{S.M} above, and note a plausible generalization
of~\eqref{d.Dubrovin} which, if true, would imply
that the $\!kG\!$-module $k((G))$ satisfies~\eqref{d.cl_exch_iff}.

Incidentally, Dubrovin \cite{Dubrovin} assumes $G$
{\em left}-ordered and regards $k((G))$ as a left $\!kG\!$-module; but
the results for left- and right-ordered groups are clearly equivalent.
He also develops his result with $k$ a general division ring,
and with the multiplication of the group
ring skewed by an action of $G$ on $k$ by automorphisms.
In this note, I restrict attention to the case where $k$ is
a field and $G$ centralizes $k,$ simply because
the added generality would be a distraction.
But the generalizations
mentioned seem to involve no fundamental complications, so
if further results are eventually obtained for the
case discussed here, the techniques are likely
to go over to those more general cases.

(Dubrovin also proves in~\cite{Dubrovin_cuts}, \cite{Dubrovin_GL2}
that group rings of certain particular classes of
right-ordered groups are embeddable in division rings;
but we are here concerned with what one can hope to prove for
arbitrary right-ordered groups.)

So let $G$ be a right ordered group and $x$ an element of $kG-\{0\}.$
The easy half of~\eqref{d.Dubrovin} is
that the action of $x$ on $k((G))$ is one-to-one,
i.e., that for any $a\in k((G))-\{0\}$ we have $ax\neq 0.$
To see this, let $g_0$ be the least element of the support of $a.$
Then for each $h\in\supp(x),$ since right multiplication
by $h$ preserves the order of $G,$ the least
element of $\supp(a\,h)$ is $g_0\,h.$
Since left multiplication by $g_0$ is one-to-one on $G,$
the finitely many well-ordered sets $\supp(a\,h)$ $(h\in\supp(x))$
have distinct least elements $g_0\,h;$ so the least of these
least elements appears exactly once when we evaluate $a\,x.$
Hence $a\,x\neq 0.$

Note, however, that
we cannot say a priori which $h\in\supp(x)$ will make $g_0\,h$
the least element of the support of $a\,x.$
It need not be the least element of $\supp(x),$
since {\em left} multiplication by $g_0$ does not, in general, preserve
the ordering of $G.$

Nevertheless, given $x\in kG-\{0\},$
the function associating to each $g\in G$
the least product $g\,h$ for $h\in\supp(x)$
will be an order-preserving bijection $G\to G.$
To see that it is order-preserving and one-to-one, let $g_0<g_1\in G,$
and take an $a\in k((G))$ with $\supp(a)=\{g_0,g_1\}.$
Then our observation that the least element of $\supp(a)\cdot\supp(x)$
has the form $g_0\,h,$ and occurs just once, shows that
$g_0\,h\in g_0\,\supp(x)$ must be distinct from
the least element of $g_1\,\supp(x),$ and less than it, as asserted.

To see that the function $G\to G$ of the
preceding paragraph is also surjective,
take any $g\in G,$ which we wish to show is in its range.
Let $h_0$ be the member of $\supp(x)$ that {\em maximizes}
$g\,h_0^{-1}.$
Note that $g\in(g\,h_0^{-1})\,\supp(x);$
I claim $g$ is the smallest element of that set.
For taking any $h_1\neq h_0$ in $\supp(x),$ by choice of $h_0$ we have
$g\,h_1^{-1}<g\,h_0^{-1},$ hence,
right multiplying by $h_1,$ we get $g<(g\,h_0^{-1})\,h_1,$ so $g$ is
indeed the least element of $(g\,h_0^{-1})\,\supp(x),$ as claimed.

Let us give the function we have defined a name.
\begin{equation}\begin{minipage}[c]{35pc}\label{d.rho}
Suppose $x\in kG-\{0\}.$
For each $g\in G,$ we shall write
$\rho_{\supp(x)}(g)$ for the least element of $g\cdot\supp(x).$
Thus, as shown
above, $\rho_{\supp(x)}$ is an order-preserving bijection $G\to G.$
\end{minipage}\end{equation}
Here $\rho_{\supp(x)}$ is mnemonic for the fact that the operation
involves {\em right} multiplication by $\supp(x).$

We can now approach the task of
showing that right multiplication by $x\in kG-\{0\}$
is surjective as a map $k((G))\to k((G)).$
Given $a\in k((G)),$ we want to construct
$b\in k((G))$ such that $b\,x=a.$
If $a=0$ there is no problem; if not, let $g_0$ be the
least element of $\supp(a).$
From the above discussion, we see
that the least element of $\supp(b)$ has to be
$\rho_{\supp(x)}^{-1}(g_0).$
Attaching to this the appropriate coefficient in $k,$
we get a first approximation to $b;$ an element $b_0\in k((G))$
whose product with $x$ has the correct lowest term.

Now let $a_1=a-b_0\,x.$
If this is zero, we are again done; if not, we let $g_1$
be the least element of its support, and repeat the process.

But can we continue this process transfinitely?
When we come to a limit ordinal $\alpha,$ will the
expression $b_\alpha$ that the previously
constructed expressions $b_\beta$ $(\beta<\alpha)$
converge to have well-ordered support?

Dubrovin shows by a transfinite induction that this
does indeed hold at every step.
As a variant approach, we shall recall in the next section a
general result on ordered sets, going back to
G.\,Higman, using which we can obtain a
well-ordered subset $Y\subseteq G$ such that the process
sketched above keeps the supports of the $b_\alpha$ within $Y.$
What we have looked at as a process of
successively modifying elements $\!b_\alpha\!$ then becomes
a transfinite coefficient-by-coefficient calculation
of $b,$ indexed by the well-ordered set~$Y.$

The result on ordered sets is quite
powerful, so we can hope that it will also
be applicable to studying solutions to several linear equations
in several unknown elements of $k((G)),$ as might be needed
to combine the idea of Dubrovin's result with the approach
of~\S\S\ref{S.mtrd}-\ref{S.M}.

\section{Generating well-ordered sets}\label{S.wo}

What sort of order-theoretic result do we need?
Given $a\in k((G))$ and $x\in kG,$ we want to modify
the former by subtracting off a multiple of the latter
having the same least term, and iterate this process.
At steps after the first, the least element in the
support of our modified $a$ might be one of the elements
of the original support, or an element in the support of
one of the terms we have subtracted off.
What we can say is that it lies in
the closure of $\supp(a)$ under adjoining, for every element
$g$ that at some stage is in our set, all the other elements of the
unique left multiple of $\supp(x)$ having $g$ for its least member.
We want to know that this closure, like $\supp(a),$ is well-ordered.

If $\supp(x)$ has $n$ elements, then the above
construction can be thought of as closing $\supp(a)$ under $n(n-1)$
{\em partial} functions.
Indeed, given $g\in G,$ and distinct
elements $h_0,\,h_1\in\supp(x),$ we are interested in
$g\,h_0^{-1} h_1$ {\em if} the left translate of
$\supp(x)$ which has $g$ as its least member is
$g\,h_0^{-1}\supp(x),$ the translate in which $h_0$ is carried to $g.$
So to each pair $h_0\neq h_1$ of elements of $\supp(x),$
let us associate the partial function $G\to G$ which, if $g$
is the least element of $(g\,h_0^{-1})\,\supp(x),$ takes $g$ to
$g\,h_0^{-1}\,h_1,$ but is undefined otherwise.

These $n(n-1)$ partial functions are partial {\em unary}
operations on $G;$ but the order-theoretic arguments to be used
can in fact handle partial operations of arbitrary finite arities.
This suggests a general formulation that would start with a well-ordered
set of elements (corresponding to $\supp(a)),$
and a finite family of partial finitary operations.
But the members of the former set can be thought of as zeroary
operations;
so if we allow our {\em finitely many} finitary operations to be
replaced by {\em well-ordered families} of such operations -- one such
family for each of finitely many arities -- then
we can treat the given set of elements as one of these families.

Finally, the assumption that the set on which we are operating
(in our case, $G)$ is given with a total ordering, and our families
of operations are indexed by well-ordered sets, can be weakened to
make the given set partially ordered, and the families of operations
``well-partially-ordered'', i.e., having
descending chain condition and no infinite antichains.
Indeed, in the case we are interested in,
there is no natural order to put on the
pairs $(h_0,h_1)$ indexing our unary partial operations;
and though $G$ is totally ordered, if we hope to generalize
our result from single relations $a\,x=b$ to families
of relations on tuples $(a_1,\dots,a_n)\in k((G))^n,$ then
the elements $a_1,\dots,a_n$ will be
multiplied by different elements of $kG,$ so it would make
most sense to regard their supports as belonging
to a union of $n$ copies of $G,$ each ordered as $G$ is,
but with elements of the different copies incomparable.

A result of the sort suggested
above was proved by G.\,Higman \cite{Higman},
except that he started with everywhere-defined functions.
However, his proof goes over without modification to partial functions.
We state the result, so generalized, and in modern language, below.
(We drop a different sort of generality in the formulation
of Higman's result.
Namely, where we assume a partial ordering on $X,$
he only assumed a preordering.
But the hypotheses of his result imply that each of his operations,
when applied to elements equivalent
under the equivalence relation determined by
the preordering (cf.~\cite[Proposition~5.2.2]{245}),
gives equivalent outputs.
From this it can be
deduced that his conclusion about the preordered set is equivalent
to the corresponding statement about the partially ordered set
gotten by dividing out by that equivalence relation.)

We shall denote the action of a partial function by ``$\po$''.

\begin{theorem}[{after G.\,Higman \cite[Theorem~1.1]{Higman}}]\label{T.Higman}
Let $X$ be any partially ordered set,
let $I_0,\dots,I_{N-1}$ be well-partially-ordered sets for some
$N\geq 0,$ and suppose that for each $n\in\{0,\dots,N{-}1\}$
we are given a partial function $s_n:X^n\times I_n\po X.$
Suppose further that each $s_n$ is, on the one
hand, {\em isotone} \textup{(}i.e., if $p,q\in X^n\times I_n$
lie in the domain of $s_n,$
and $p\leq q$ under coordinatewise comparison,
then $s_n(p)\leq s_n(q)),$ and, on the other hand,
{\em nondecreasing} in its arguments in $X$ \textup{(}i.e., if
$s_n$ is defined at $p=(x_0,\dots,x_{n-1},i)\in X^n\times I_n,$
then $s_n(p)\geq x_m$ for all $m<n).$

Let $Y$ be the subset of $X$ generated by the above
operations; that is, the least subset of $X$ with the property that
$s_n(x_0,\dots,x_{n-1},i)\in Y$ whenever $0\leq n<N,$
$x_0,\dots,x_{n-1}\in Y,$ $i\in I_n,$ and $s_n$
is defined on $(x_0,\dots,x_{n-1},i).$
Then $Y$ is well-partially-ordered.

\textup{(}Remark: if $I_0$ is empty, then $Y$ is empty.
It is the elements $s_0(i)$ $(i\in I_0)$
that ``start'' the process that generates~$Y.)$
\end{theorem}

\begin{proof}
As in \cite{Higman}.
\end{proof}

I recommend Higman's proof as a tour-de-force worth reading.
(His Theorem~2.6, used in that proof, is a method of induction
over the class of all $\!n\!$-tuples of well-partially-ordered sets.)

\section{Recovering Dubrovin's bijectivity result}\label{S.bij_via_Higman}

Let us now, with the help of the above result,
prove the bijectivity of the right action
on $k((G))$ of every nonzero element of $kG.$

Given
\begin{equation}\begin{minipage}[c]{35pc}\label{d.x,a}
$a\in k((G))$\quad and\quad $x\in kG-\{0\},$
\end{minipage}\end{equation}
we wish to find $b\in k((G))$ such that $bx=a.$
With this goal, we start by applying Theorem~\ref{T.Higman}
with $N=2,$ and the following choices of $X,$
$I_n$ and $s_n$ $(n<2):$
\begin{equation}\begin{minipage}[c]{35pc}\label{d.X,N}
$X=G,$ with its given right ordering.
\end{minipage}\end{equation}
\begin{equation}\begin{minipage}[c]{35pc}\label{d.I_0}
$I_0=\supp(a),$ with $s_0:I_0\to X$ given by (the restriction
to $I_0$ of) $\rho_{\supp(x)}^{-1},$ defined in~\eqref{d.rho}.
Thus, $s_0$ takes each $g\in\supp(a)$ to the
$g'\in G$ such that $g$ is the least element of $g'\,\supp(x).$
\end{minipage}\end{equation}
\begin{equation}\begin{minipage}[c]{35pc}\label{d.I_1}
$I_1=$ the finite set $\{(h_0,h_1)\in\supp(x)^2\mid h_0\neq h_1\},$
given with the antichain ordering (making distinct elements
incomparable), and $s_1: G\times I_1\po G$ is defined by
$s_1(g,(h_0,h_1))=g\,h_0^{-1}h_1$
if $g$ is the least element of $g\,h_0^{-1}\,\supp(x)$
(equivalently, if $\rho_{\supp(x)}^{-1}(g)=g\,h_0^{-1}),$ and
is undefined otherwise.
\end{minipage}\end{equation}

Thus, $s_1$ encodes the $n(n-1)$ partial functions discussed in the
second paragraph of the preceding section.

To see that the hypotheses of Theorem~\ref{T.Higman} are satisfied,
note that $I_0$ and $I_1$ are well-partially-ordered, the former
because, being the support of an element of $k((X)),$
it is well-ordered, the latter because, though
an antichain, it is finite.
Since $s_0$ has no arguments in $X,$
it only needs to be isotone in its argument in $I_0,$
which it is, because $\rho_{\supp(x)}^{-1}$ is an
order-automorphism of $G.$
Since $I_1$ is an antichain, $s_1$ need only
be isotone and non-decreasing in its argument in $G.$
It is isotone in that argument because it is given, when defined,
by {\em right} multiplication by the element $h_0^{-1}h_1.$
By its definition, it is non-decreasing
(in fact, increasing) in that argument when defined.

Thus, Theorem~\ref{T.Higman} yields a
subset $Y\subseteq G$ closed under the above operations
and well-partially-ordered; which,
since $G$ is totally ordered, means well-ordered.
Closure under the operation of~\eqref{d.I_0} means that
$\rho_{\supp(x)}^{-1}(\supp(a))\subseteq Y,$
while closure under the operations of~\eqref{d.I_1}
says that for each $g\in Y$ we also have
$\rho_{\supp(x)}^{-1}(g)\cdot\supp(x)\subseteq Y.$
Note, finally, that the definition~\eqref{d.rho} of $\rho_{\supp(x)}$
shows that
\begin{equation}\begin{minipage}[c]{35pc}\label{d.a-b'x}
For every $b'\in k((G))$ having support in $Y,$
we have $\supp(a-b'x)\subseteq \rho_{\supp(x)}(Y).$
\end{minipage}\end{equation}

Let us now construct by recursion elements $\beta_g\in k$
for all $g\in Y,$ such that $(\sum_{g\in Y}\beta_g\,g)\,x=a.$
To do this, assume recursively that for some $g\in Y$ we have found
$\beta_{g'}$ for all $g'<g$ in $Y,$ such that for each $g_0<g,$
\begin{equation}\begin{minipage}[c]{35pc}\label{d.g_0_approx}
all elements of\,
$\supp(a-(\sum_{g'\in Y,\, g'\leq g_0}\,\beta_{g'}\,g')\,x\,)$
are $>\rho_{\supp(x)}(g_0).$
\end{minipage}\end{equation}
Then applying~\eqref{d.g_0_approx} to the greatest $g_0<g$ in $Y$ if
there is one, or passing to the ``limit'' gotten by taking
the union of the ranges of summation in~\eqref{d.g_0_approx}
if there is not, we can say that
\begin{equation}\begin{minipage}[c]{35pc}\label{d.approx_lim}
all elements of\,
$\supp(a-(\sum_{g'\in Y,\, g'<g}\,\beta_{g'}\,g')\,x\,)$
are $>\rho_{\supp(x)}(g_0)$ for all $g_0<g$ in $Y.$
\end{minipage}\end{equation}
(The two changes from~\eqref{d.g_0_approx} are in the
range of summation, and the final quantification of $g_0.)$
Since $\rho_{\supp(x)}$ is an order automorphism of $G,$
we see from~\eqref{d.a-b'x} that
the condition ``\!$>\rho_{\supp(x)}(g_0)$ for all $g_0<g$ in $Y\!$''
is equivalent to ``$\geq \rho_{\supp(x)}(g)$'', so~\eqref{d.approx_lim}
says
\begin{equation}\begin{minipage}[c]{35pc}\label{d.approx}
all elements of\,
$\supp(a-(\sum_{g'\in Y,\, g'<g}\,\beta_{g'}\,g')\,x\,)$
are $\geq \rho_{\supp(x)}(g).$
\end{minipage}\end{equation}

Given~\eqref{d.approx} for some $g\in Y,$ let $\gamma\in k$ be the
coefficient of $\rho_{\supp(x)}(g)$ in
$a-(\sum_{g'\in Y,\, g'<g}\,\beta_{g'}\,g')\,x,$
let $h=g^{-1}\,\rho_{\supp(x)}(g)\in\supp(x),$
and let $\delta$ be the coefficient of $h$ in $x.$
Then letting $\beta_g=\gamma\,\delta^{-1},$ we see that
this is the unique choice of coefficient for $g$ that
will lead to~\eqref{d.g_0_approx} holding with $g$ in place of $g_0.$

Constructing $\beta_g$ in this way for each $g\in Y,$
and taking $b=\sum_{g\in Y}\,\beta_g\,g,$ we
find that $bx=a,$ as desired.
From the way our recursive construction has forced a unique value
for each $\beta_g$ $(g\in Y),$ it is not hard to deduce
that $b$ is unique for that
property (though this uniqueness is most easily seen
as in~\S\ref{S.bij}).
Thus we have

\begin{theorem}[after N.\,I.\,Dubrovin \cite{Dubrovin}]\label{T.Dubrovin}
If $G$ is a right-ordered group, and $x\in kG-\{0\},$ then the
action of $x$ on the right $\!kG\!$-module $k((G))$ is bijective.\qed
\end{theorem}

\section{What should we try to prove next?}\label{S.further?}

In the context of Theorem~\ref{T.Dubrovin}, consider any column vector
$x=\left(\begin{matrix} x_1 \\
x_2 \end{matrix}\right)\in(kG)^2,$ and assume
for simplicity that both $x_1$ and $x_2$ are nonzero.
From that theorem it is not hard
to deduce that the set $K$ of row vectors $a=(a_1,a_2)\in k((G))^2$
right-annihilated by $x$ has the property that the projection maps
to first and second components each give a bijection $K\to k((G)).$
We may ask
\begin{question}\label{Q.next_case}
Given $x_1,\,x_2\in kG-\{0\},$ let $K$ be, as above, the kernel of
the map $k((G))^2\to k((G))$ induced by the column vector
$\left(\begin{matrix} x_1 \\
x_2 \end{matrix}\right)\in(kG)^2.$
Is it true that
for each $y=\left(\begin{matrix} y_1 \\
y_2 \end{matrix}\right)\in(kG)^2,$ the map $K\to k((G))$ induced
by $y,$ taking $(a_1,a_2)\in K$ to $a_1 y_1 + a_2 y_2,$
is either zero or bijective?
\end{question}

A positive answer seems intuitively plausible.

We shall see in the next section a sequence of conditions,
indexed by an integer $n\geq 1,$ on a right
module $M$ over a general ring $R,$ such that for $R=kG$ and $M=k((G)),$
the result of Theorem~\ref{T.Dubrovin} is the $n=1$ case,
a positive answer to Question~\ref{Q.next_case} would
be the $n=2$ case, and the full set of conditions
would imply that $R$ is embeddable in a division ring.

For the moment, let us restrict attention to Question~\ref{Q.next_case}.
To see concretely what it asks, note that the kernel $K$
can be described as the set of elements of $k((G))^2$
of the form $(a_1, -a_1 x_1 x_2^{-1})$ (where by $x_2^{-1}$
I mean the inverse of the action of $x_2$ on $k((G))\,).$
The image of such a member of $K$ under $y$
is $a_1 y_1 - a_1 x_1 x_2^{-1} y_2;$
hence a positive answer to Question~\ref{Q.next_case} would say
that in the ring of endomaps of $k((G))$ generated
by the actions of elements of $kG$ and the inverses of those actions,
every map of the form $y_1 - x_1 x_2^{-1} y_2$ is
either zero or invertible.
So this is, indeed, a ``next step''
after the invertibility of the actions of
nonzero elements of $kG$ itself.

If $y_2\neq 0,$ then right multiplying
$y_1 - x_1 x_2^{-1} y_2$ by $y_2^{-1},$ we see that a positive answer
to Question~\ref{Q.next_case} is also
equivalent to the statement that every nonzero map of the
form $y_1 y_2^{-1} - x_1 x_2^{-1}$ is invertible.
Alternatively, left multiplying by $x_1^{-1}$ gives
the corresponding condition on $x_1^{-1} y_1 - x_2^{-1} y_2,$
while if we instead right multiply by $y_1^{-1},$ we get the
same statement for $1 - x_1 x_2^{-1} y_2\,y_1^{-1}.$
So we can restate Question~\ref{Q.next_case} as
\begin{question}\label{Q.next_case_alt}
Is it true that for all $x_1,\,x_2,\,y_1,\,y_2\in kG-\{0\},$
the endomap of $k((G))$ given by the \textup{(}right\textup{)}
action of $y_1 y_2^{-1} - x_1 x_2^{-1}$ is either zero or invertible?
Equivalently, is the same true of the endomaps given
by the actions of $x_1^{-1} y_1 - x_2^{-1} y_2,$
$y_1 - x_1 x_2^{-1} y_2,$
$1 - x_1 x_2^{-1} y_2\,y_1^{-1}$?
\end{question}

(The assumption that $y_1$ and $y_2$ are nonzero was not made in
Question~\ref{Q.next_case}; but if either or both is zero,
an affirmative answer to Question~\ref{Q.next_case} is easily
deduced from Theorem~\ref{T.Dubrovin} and our description of~$K.)$

The difficulty in approaching
Question~\ref{Q.next_case_alt} is that we have no evident test
for when a map such as $x_1^{-1} y_1 - x_2^{-1} y_2$ should be zero.
Since each of $x_1,$ $x_2,$ $y_1,$ $y_2$ is a finite $\!k\!$-linear
combination of elements of $G,$ we might hope that this could
be answered by some finite computation; but the inverses
appearing in the expressions in
Question~\ref{Q.next_case_alt}
represent operators on $k((G))$ that can behave differently
on different terms of an element of that module.

A variant of this difficulty:  Note
that for every $r\in kG,$ the expression
$x_1^{-1} y_1 - x_2^{-1} y_2$ has the same action as
$x_1^{-1} (y_1 + x_1 r) - x_2^{-1} (y_2 + x_2 r).$
This can allow one to transform an expression
$a (x_1^{-1} y_1 - x_2^{-1} y_2)$ $(a\in k((G))\,)$
such that the lowest elements of the supports of
$a(x_1^{-1} y_1)$ and $a(x_2^{-1} y_2)$ cancel one
another to an expression for the same element with higher lowest terms.
But an $r$ that has that effect
for one $a\in k((G))$ might have the opposite
effect for an $a'$ with a different lowest term.

We remark, as an aside, that the statement that every map
of the form $z_1^{-1} + z_2^{-1}$ $(z_1,z_2\in kG-\{0\})$ is zero
or invertible does not require an answer to
Question~\ref{Q.next_case_alt}; it
follows from Theorem~\ref{T.Dubrovin}, by looking at
$z_1^{-1} + z_2^{-1}$ as $z_1^{-1} (z_2 + z_1) z_2^{-1}.$
(More generally, this holds for every map of the form
$z_1^{-1} w_1 + w_2 z_2^{-1},$ since this can be rewritten
$z_1^{-1} (w_1 z_2 + z_1 w_2)z_2^{-1}.)$
The statement that every map of the
form $z_1^{-1} + z_2^{-1} + z_3^{-1}$ is zero
or invertible would, on the other hand, follow
from a positive answer to
Question~\ref{Q.next_case_alt}, by writing that
sum as $z_1^{-1} ((z_2 + z_1) + z_1 z_3^{-1} z_2) z_2^{-1},$
and regarding the parenthesized factor as having the form
$y_1 - x_1 x_2^{-1} y_2.$

\section{A general condition}\label{S.either/or}

Let us now give the promised family of conditions generalizing
both the result proved by Dubrovin and the extension of
that result asked for in Question~\ref{Q.next_case_alt}.

(Readers who have read \S\ref{S.n<N} will notice similarities
between the properties treated there and those considered below;
for instance, between Lemma~\ref{L.strong_n-1}(vii) and
Lemma~\ref{L.X'}.
The material of \S\ref{S.n<N}
was, in fact, motivated by the idea of abstracting
the results below to closure
operators not necessarily arising from modules.
But that turned out to require lengthier arguments,
and I have not tried to carry it to completion.)

Given a ring $R,$ a right $\!R\!$-module $M,$ and any
$n\geq 0,$ let us refer to an $n\times n$ matrix
over $R$ as \mbox{{\em $\!M\!$-invertible}}
if it induces an invertible map $M^n\to M^n.$
The condition that we will be interested in for general $n\geq 1$ is
\begin{equation}\begin{minipage}[c]{35pc}\label{d.either/or}
For every $n\times n{-}1$ matrix $X$ over $R$
whose top $n{-}1\times n{-}1$ block is $\!M\!$-invertible, and
every column vector $y\in R^n,$ the action of $y$ either
annihilates the kernel $K$ of the additive
group homomorphism $M^n\to M^{n-1}$ induced by $X,$ or maps $K$
bijectively onto $M.$
\end{minipage}\end{equation}

What does the $n=1$ case of~\eqref{d.either/or} say?
In that case the matrix $X$ of~\eqref{d.either/or} is
$1\times 0,$ hence represents the unique map $M\to M^0=\{0\},$
which has kernel $M.$
Its upper $0\times 0$ block, also an empty matrix, represents
the unique endomorphism of $M^0=\{0\},$ which is clearly
invertible; i.e., that block is $\!M\!$-invertible.
Thus, that case of~\eqref{d.either/or} says
that every $y\in R$ not annihilating $M$ maps $M$ bijectively
to itself.
In particular, for $R=kG$ and $M=k((G)),$ this is the
statement of Theorem~\ref{T.Dubrovin}; so for these $R$ and $M,$
the $n=1$ case of~\eqref{d.either/or} holds.

Given Theorem~\ref{T.Dubrovin}, we now see that
Question~\ref{Q.next_case} asks whether the $n=2$ case
of~\eqref{d.either/or} holds for this ring and module.
(The one detail that has to be cleared up is that in formulating
Question~\ref{Q.next_case}, I assumed for conceptual simplicity that
$x_2\neq 0,$ which is not in the hypothesis of~\eqref{d.either/or}.
But the hypothesis of~\eqref{d.either/or} for $n=2$ does
imply $x_1\neq 0,$ hence if $x_2=0,$ we get $K=\{0\}\times M,$ and
the desired result holds by Theorem~\ref{T.Dubrovin}.
So Question~\ref{Q.next_case} asks for the part
of the $n=2$ case of~\eqref{d.either/or} which does not
follow from this observation.)

In studying the general case of~\eqref{d.either/or},
the following observation will be useful.

\begin{lemma}\label{L.X'}
Let $R$ be a ring, $M$ a right $\!R\!$-module, $X$
an $n\times n{-}1$ matrix over $R$ for some $n>0,$
and $y\in R^n$ a column vector.
Then if the upper $n{-}1\times n{-}1$ block of $X$
is $\!M\!$-invertible, and $y$ maps $\r{ann}_{M^n}(X)$ bijectively
to $M,$ then the $n\times n$ matrix $X'$ gotten by
appending $y$ to $X$ as an $\!n\!$-th column is $\!M\!$-invertible.
\end{lemma}

\begin{proof}
By assumption, $y$ annihilates no member of $\r{ann}_{M^n}(X);$
clearly this says that the matrix $X'$
annihilates no member of $M^n.$

To see that $X'$ is surjective, let $a\in M^n$
be an element we want to show is in its range.
By the $\!M\!$-invertibility of the top $n{-}1\times n{-}1$ block
of $X,$ we can find $b\in M^n$ with last term $0,$ and whose
first $n-1$ terms form a vector carried by that subblock of $X$
to the first $n-1$ terms of $a.$
Since the last term of $b$ is $0,$ multiplying $b$ by the whole
matrix $X$ still gives the first $n-1$ terms of $a.$
If we apply $y$ to $b,$ we get an element $by\in M$
which may differ from the desired last term $a_n$ of $a;$
but since $y$ carries the annihilator of $X$ bijectively
to $M,$ we can find an element $b'$ in that
annihilator which is carried by $y$ to $a_n - by.$
We then get $(b+b')X'=a,$ proving surjectivity.
\end{proof}

From this, we can prove, for any $N\geq 0,$

\begin{lemma}\label{L.max_inv}
Suppose $R$ and $M$ are a ring and module
satisfying~\eqref{d.either/or} for all $0<n\leq N.$
Let $H$ be an $n\times n'$ matrix over $R$ with $n\leq N,$
and suppose \textup{(}as we may, without loss of generality,
by a permutation of the rows and columns\textup{)} that
$H=\left(\begin{matrix} A\ B \\
C\ D \end{matrix}\right),$ where $A$ is maximal among
$\!M\!$-invertible square submatrices of $A.$
Say $A$ is $m\times m.$
\textup{(}Here one or more of $m,$ $n-m,$ $n'-m$ may be zero, making
some of the submatrices $A,$ $B,$ $C,$ $D$ empty.\textup{)}

Then every column of
$\left(\begin{matrix} B \\
D \end{matrix}\right)$ annihilates
$\r{ann}_{M^n}(\left(\begin{matrix} A \\
C \end{matrix}\right));$
equivalently, $\r{ann}_{M^n}(H)=
\r{ann}_{M^n}(\left(\begin{matrix} A \\
C \end{matrix}\right)).$
\end{lemma}

\begin{proof}
Let us write $A^{-1}$ for the inverse of the action
of $A$ on $M^m$ (though this
action is not in general represented by a matrix over $R).$
Then every element of $\r{ann}_{M^n}(\left(\begin{matrix} A \\
C \end{matrix}\right))$ is uniquely determined by
its final $n-m$ entries; namely,
given $b\in M^{n-m},$ one sees that $(-b\,C A^{-1}, b)$
is the unique member of $M^n$ ending in $b$ and
annihilated by $\left(\begin{matrix} A \\
C \end{matrix}\right).$
Hence $\r{ann}_{M^n}(\left(\begin{matrix} A \\
C \end{matrix}\right))$ is the direct sum, over all
$i$ with $m<i\leq n,$ of the additive subgroup of that
annihilator consisting of
elements whose only nonzero entry after the first $m$ entries
(if any) is in the $\!i\!$-th position.

Now suppose that for some $j>m,$ the $\!j\!$-th column of $H$
did not annihilate $\r{ann}_{M^n}(\left(\begin{matrix} A \\
C \end{matrix}\right)).$
By rearranging the columns of $H$ after the $\!m\!$-th,
we can assume without loss of generality that $j=m+1.$
Since the $\!m{+}1\!$-st column of $H$ does not annihilate
$\r{ann}_{M^n}(\left(\begin{matrix} A \\
C \end{matrix}\right)),$ that column will not annihilate all of
the direct summands mentioned in the preceding paragraph, and by a
rearrangement of the rows of $H,$ we can assume
that a summand which it fails to annihilate consists of the members of
$\r{ann}_{M^n}(\left(\begin{matrix} A \\
C \end{matrix}\right))$ whose only nonzero entry after
the $\!m\!$-th (if any) is the $\!m{+}1\!$-st.

Let us now apply~\eqref{d.either/or}, putting in the role of $X$
the $m{+}1\times m$ matrix consisting of $A$ and the top row of $C,$
and in the role of $y$ the column vector consisting of the
first $m+1$ entries of the $\!m{+}1\!$-st column of $H.$
By assumption, that column does not annihilate the
annihilator of that matrix.
By Lemma~\ref{L.X'}, that makes the upper left $m{+}1\times m{+}1$
submatrix of $H$ an $\!M\!$-invertible matrix,
contradicting the maximality assumption on $A.$
This contradiction shows that every column of
$\left(\begin{matrix} B \\
D \end{matrix}\right)$ annihilates
$\r{ann}_{M^n}(\left(\begin{matrix} A \\
C \end{matrix}\right)),$ as claimed.
\end{proof}

We can now prove

\begin{proposition}\label{P.e/o=>exch}
Let $R$ and $M$ be a ring and module satisfying~\eqref{d.either/or}
for $0<n\leq N.$
Then condition~\eqref{d.cl_exch_iff_mx-}
holds for all $n\leq N.$

Hence if~\eqref{d.either/or} holds for all $n>0,$ and
the $\!R\!$-module $M$ is faithful, then $R$ admits an
embedding in a division ring.
\end{proposition}

\begin{proof}
Assume, by way of contradiction, that strict inclusions
as in~\eqref{d.cl_exch_iff_mx-} hold; however, let us write $H$
for the matrix there called $A,$ and write the matrices
there called $B$ and $C$ as $(H,s)$
and $(H,t),$ where $s,t\in R^n$ (freeing up the
letters $A$ through $D$ for use as in Lemma~\ref{L.max_inv}).

Applying Lemma~\ref{L.max_inv} to $H,$ we get (after rearranging the
rows and columns of $H)$
an $n\times m$ submatrix $H'=\left(\begin{matrix} A \\
C \end{matrix}\right),$ where $m\leq n-1,$
having the same left annihilator in $M^n$
as $H,$ and such that $A$ is invertible.
By hypothesis, the vectors $s$ and $t$ each act nontrivially
on $\r{ann}_{M^n}(H)=\r{ann}_{M^n}(H'),$
with $s$ having strictly larger annihilator
there than $t$ does, so
\begin{equation}\begin{minipage}[c]{35pc}\label{d.H',sH',t}
$\r{ann}_{M^n}(H',s)\ \supsetneqq\ \r{ann}_{M^n}(H',t).$
\end{minipage}\end{equation}

Now as in the proof of Lemma~\ref{L.max_inv}, we see that in
$\r{ann}_{M^n}(H'),$ each element is determined uniquely
by its final $n-m$ terms, and that since $s$ acts nontrivially
on that annihilator, it will act nontrivially on an element
in which only one of those positions has a nonzero entry.
By another rearrangement of rows we can assume that
that position is the $\!m{+}1\!$-st.
The annihilator of the action $t$ on $\r{ann}(H')$
was assumed to be contained in that of $s,$ so $t$ will
also act nontrivially on that element.

Hence by Lemma~\ref{L.X'}, in each of the $n\times m{+}1$ matrices
$(H',s)$ and $(H',t),$ the top $m{+}1\times m{+}1$ block
will be invertible.
Hence in the annihilators of those matrices, every element
will be determined uniquely by its last $n-m-1$ terms.
Clearly, if the functions determining such elements
from their last $n-m-1$ terms are the same
for $(H',s)$ and $(H',t),$ then the annihilators of those matrices
are the same, while if the functions are different,
those annihilators are incomparable; so neither possibility
is compatible with the assumed strict inclusion~\eqref{d.H',sH',t}.

This contradiction completes the proof of~\eqref{d.cl_exch_iff_mx-}.
The final assertion of the
proposition follows by Lemma~\ref{L.cl_fr_M}.
\end{proof}

We remark that an $\!R\!$-module $M$
satisfying~\eqref{d.cl_exch_iff_mx-} for all $n,$
and hence leading to a homomorphism from $R$ to a division
ring $D,$ need not, in general, itself be a vector space
over a division ring.
For example, if $R$ is a commutative integral domain, one
finds that the choice $M=R$ leads as in~\S\ref{S.M}
to a closure operator that gives the field of fractions $F$ of $R.$
Indeed, the closure operator determined by $M$ is the same
as that determined by the $\!R\!$-module $F,$ since
the annihilator of any row vector over $F$
is also the annihilator of a row vector over $R,$
gotten by clearing denominators.
On the other hand,
I do not know whether a module satisfying the stronger
condition~\eqref{d.either/or} for all $n$ must be a vector space over
the division ring $D$ that it determines.
(This is indeed so in the case of commutative $R,$ where the
$n=1$ case of~\eqref{d.either/or} is the analog of Dubrovin's result:
It says that every element of $R$ that does not annihilate $M$
acts invertibly on it.)

The $n=2$ case of~\eqref{d.either/or},
discussed in the preceding section, should be a useful
test case for ideas on how to try to prove
that for $R=kG$ and $M=k((G)),$~\eqref{d.either/or} holds for all $n.$

\section{Sandwiching $kG$ between a right and a left module}\label{S.G*}

For $R$ an algebra over a field $k$ and $M$ a right $\!R\!$-module,
the dual $\!k\!$-vector space $M^*=\r{Hom}_k(M,k)$ has
a natural structure of {\em left} $\!R\!$-module.
If we write the image of $a\in M$ under
$b\in M^*$ as $\lang a,b\rang\in k,$ then the relation between these
module structures is described by the rule
\begin{equation}\begin{minipage}[c]{35pc}\label{d.MrM*}
$\lang a\,r,\,b\rang\ =\ \lang a,\,r\,b\rang$\quad
for\quad $a\in M,\ r\in R,\ b\in M^*.$
\end{minipage}\end{equation}
I do not know whether $M^*$ can somehow be used, together with
$M,$ in studying whether $R$ is embeddable in a division ring.
However, the above observation is really a lead-in to the observation
that for $R=kG$ and $M=k((G)),$ there is
a left $\!R\!$-module that behaves much like the above $M^*,$ but is
not itself constructed from $M,$ and hence has (conceivably) a better
chance of bringing additional strength to our investigations.

Namely, given a group $G$
with a right-invariant ordering $\leq,$ let $G^*$ be the same
group under the corresponding {\em left}-invariant
ordering, $\leq^*,$ characterized by
\begin{equation}\begin{minipage}[c]{35pc}\label{d.G*}
$g\,\leq^*\,h\quad\iff\quad g^{-1}\geq h^{-1}.$
\end{minipage}\end{equation}

(A right- or left-invariant ordering $\leq$
on a group is determined by its
{\em positive cone}, $\{g\in G\mid g\geq 1\}.$
The ordering $\leq^*$ defined above is the left-invariant ordering
having the same positive cone as the given
right-invariant ordering ${\leq}.$
Indeed, writing $P$ for the positive cone of ${\leq},$ we have
$g\leq h$ if and only if $h\in Pg,$ so by~\eqref{d.G*},
$g\leq^* h$ if and only if $g^{-1}\in Ph^{-1},$
which, left-multiplying by $g$ and right-multiplying by $h,$
comes to $h\in gP.)$

Let us write $k((G^*))$ for the space of
formal $\!k\!$-linear combinations of elements of $G$
having well-ordered supports under $\leq^*;$
this clearly has a natural structure of {\em left} $\!kG\!$-module.
I claim that we can define a $\!k\!$-bilinear
map $\lang\,,\,\rang:k((G))\times k((G^*))\to k$ by
\begin{equation}\begin{minipage}[c]{35pc}\label{d.kGkG*}
$\lang \sum \alpha_g g,\ \sum \beta_h h\rang\ =
\ \sum_{g\in G}\,\alpha_g \beta_{g^{-1}}$\quad for\quad
$\sum \alpha_g g\in k((G))$ and $\sum \beta_h h\in k((G^*)).$
\end{minipage}\end{equation}
To see that the right-hand side of the equation
of~\eqref{d.kGkG*} makes sense, let $A$ be the set of
$g\in G$ such that both $\alpha_g$ and $\beta_{g^{-1}}$ are nonzero.
The condition $\sum \alpha_g g\in k((G))$
shows that $A$ is well-ordered under ${\leq}.$
Similarly, since $A$ is contained in the set of {\em inverses} of
elements of the support of $\sum \beta_h h,$ and the latter support
is well-ordered under $\leq^*,$~\eqref{d.G*} shows that
$A$ is reverse-well-ordered under ${\leq}.$
Being both well-ordered and reverse-well-ordered
under $\leq,$ $A$ is finite; so the sum on the
right-hand side of~\eqref{d.kGkG*} is indeed defined.

The formula~\eqref{d.kGkG*} looks as though it says, ``Multiply
the formal sums $\sum\alpha_g g$ and $\sum\beta_h h$ together,
and take the coefficient of $1$ in the result''.
But though the summation that would give
that coefficient is, as we have just
seen, defined, the same need not be true of the coefficients
of other members of $G.$
For instance, if $G$ contains elements $s,t,$ both $>1,$ such
that $t s = s t^{-1},$ then $\sum_{i\geq 0} t^i\in k((G))$ belongs
to both $k((G))$ and $k((G^*));$ hence by left-invariance
of the order on $G^*,$
$\sum_{j\geq 0} s t^j$ also belongs to $k((G^*)).$
But the formal product of these two elements is
$(\sum_{i\geq 0} t^i)\,(\sum_{j\geq 0} s t^j)=
\sum_{i,j\geq 0} s t^{j-i},$
in which the term $s$ occurs infinitely many times.
(More generally, in this summation, each term $s t^j$
occurs infinitely many times, while terms $t^j,$
in particular, the term $1,$ never occur; which is
consistent with our observation that $1$ can occur only
finitely many times.)

Returning to the map~\eqref{d.kGkG*}, one finds that it
satisfies the analog of~\eqref{d.MrM*}:
\begin{equation}\begin{minipage}[c]{35pc}\label{d.kGrkG*}
$\lang a\,r,\,b\rang =\lang a,\,r\,b\rang$\quad for\quad
$a\in k((G)),$ $r\in kG,$ $b\in k((G^*)).$
\end{minipage}\end{equation}
This is intuitively clear from the ``coefficient of~$1$''
interpretation of $\lang\,,\,\rang.$
To verify it formally, one can first check it for $r\in G,$ then take a
finite $\!k\!$-linear combination of the resulting formulas.

Let us write the common value of the two
sides of~\eqref{d.kGrkG*} as $\lang a\,r\,b\rang.$
Thus, given $a\in k((G))$ and $b\in k((G^*)),$
though one cannot associate to each $g\in G$ the ``coefficient
of $g$ in their product'', one can associate to each such
$g$ the value $\lang a\,g^{-1}\,b\rang.$
It is not hard to check that this is in fact the
coefficient of $g$ in the formal product $ba;$
so the summations giving all coefficients of that
product (unlike the summations that would
give the coefficients in $ab)$
do each involve only finitely many terms.
Thus, the construction sending a pair $(a,b)$ to
the formal sum $\sum (\lang a g^{-1} b\rang)\,g\in k^G,$
equivalently, to the formal product $ba,$ is a
well-defined $\!k\!$-bilinear map $k((G))\times k((G^*))\to k^G.$
However, the elements of the resulting
subspace $k((G^*))\ k((G))\subseteq k^G$
are not as ``nice'' as those of $k((G))$ and $k((G^*)).$
For instance, for $G$ having positive elements satisfying
$t s = s t^{-1}$ as above, $k((G))$ contains
$(\sum_{i\geq 0} t^i)s=s(\sum_{i\geq 0} t^{-i}),$ and $k((G^*)),$
as we have noted,
contains $s(\sum_{i>0} t^i);$ so $k((G^*))\ k((G))$ will contain
$s(\sum_{i\geq 0} t^{-i})\cdot 1 + 1\cdot s(\sum_{i>0} t^i)
=s(\sum_{-\infty}^\infty t^i).$
(If one wants to see that a product
of a single element of $k((G^*))$ with
a single element of $k((G))$ can misbehave in this way, note that in
the product
$(1+s(\sum_{i\geq 0} t^{-i}))\cdot(1+s(\sum_{i>0} t^i)),$ the terms
homogeneous of degree $1$ in $s$ give the expression just described.)
However (again writing $P$ for the positive cone of the
right-ordered group $G,$ equivalently
of the left-ordered group $G^*)$
we can at least say that each element of
$k((G^*))\ k((G))$ has support which is contained in
$u\,P\,v$ for some $u,v\in G,$
equivalently, which is disjoint from $u\,(P-\{1\})^{-1}\,v.$
Namely, given $\sum_{i=1}^n b_i a_i$
with each $b_i\in k((G^*))$ and each $a_i\in k((G)),$ take $u$
such that the supports of all the $b_i$ are in $uP,$
and $v$ such that the supports of all the $a_i$ are in $Pv.$

Suppose we now let $S$ denote the set of pairs $(s_1,s_2)$
such that $s_1$ is a $\!k\!$-vector-space endomorphism
of $k((G))$ and $s_2$ a $\!k\!$-vector-space endomorphism
of $k((G^*)),$ written on the right and the left respectively,
which satisfy
\begin{equation}\begin{minipage}[c]{35pc}\label{d.kGskG*}
$\lang a\,s_1,\,b\rang =\lang a,\,s_2\,b\rang$\quad for\quad
$a\in k((G)),$ $b\in k((G^*)).$
\end{minipage}\end{equation}
It is easy to see that in such a pair, $s_1$ and $s_2$
each determine the other.
The set $S$ forms a $\!k\!$-algebra under the obvious operations,
and contains a copy of $kG,$ consisting of all
pairs $(r,r),$ where by abuse of notation we let the symbol
for $r\in kG$ denote both the right action of $r$
on $k((G))$ and its left action on $k((G^*)).$
For nonzero $r\in kG,$ we can see from Theorem~\ref{T.Dubrovin}
and its left-right dual that
all such elements are invertible in $S;$ so $S$ contains all
ring-theoretic expressions in ``elements of $kG$'' and their inverses.

But if one has any hope that $S$ might be a division ring (as
I briefly did), that is quickly squelched.
It contains, for instance, a copy of the direct
product $\!k\!$-algebra $k^G.$
Namely, if we let each $(c_g)_{g\in G}$ in that algebra
act on $k((G))$ by $\sum \alpha_g g\mapsto \sum c_g \alpha_g g$
and on $k((G^*))$ by $\sum \beta_g g\mapsto \sum c_{g^{-1}} \beta_g g,$
these actions are easily seen to satisfy~\eqref{d.kGskG*},
and to have the ring structure of the direct product of fields $k^G.$

In conclusion, I do not know whether the interaction of the right
$\!kG\!$-module $k((G)),$ the left $\!kG\!$-module $k((G^*)),$
and the operator $\lang\,,\,\rang$ may, in some
way, be useful in tackling the question of whether
$kG$ can be embedded in a division ring.

\section{Further ideas -- also having difficulties}\label{S.variants}

\subsection{A different sort of $\!kG\!$-module?}\label{SS.prod_G_mod_?}
We noted in the preceding section that a
right ordered group $G$ can have elements $s$ and $t$
satisfying $ts = st^{-1}.$
Indeed, that relation gives a presentation of
the simplest example of a group admitting
a right invariant ordering but not
a two-sided invariant ordering:
\begin{equation}\begin{minipage}[c]{35pc}\label{d.ts=}
$G\ =\ \lang s,t\mid t s = s t^{-1}\rang.$
\end{minipage}\end{equation}

If we write elements of this group in the normal
form $t^i s^j$ $(i,j\in\Z),$ it is straightforward
to verify that a right ordering is given by lexicographic
ordering of the pairs $(j,i):$
\begin{equation}\begin{minipage}[c]{35pc}\label{d.st}
$t^i s^j \leq t^{i'} s^{j'}$\quad $\iff$ \quad
$j<j',$ or $j=j'$ and $i\leq i'.$
\end{minipage}\end{equation}
(Of course, elements of $G$ also have the normal form $s^j t^i;$
but if we used that, we would have to describe our right ordering
as lexicographic ordering by the pairs $(j,\,(-1)^j i).)$

In fact, it is easy to check that~\eqref{d.st}, its conjugate
by $s,$ and their opposites,
are the only right-invariant orderings on $G.$

Now consider the element $1-t\in kG.$
We know by Theorem~\ref{T.Dubrovin} that it acts bijectively on
the right on $k((G));$
let us write the inverse of this action on $k((G))$ as $(1-t)^{-1}.$
Where does that operation send $1\in k((G))$?
Not surprisingly, to $1+t+t^2+\dots+t^n+\cdots.$
Where does it send $s$?
We might expect that to go to $s+st+st^2+\dots+st^n+\cdots;$
but rewriting the terms of this expression in our normal form,
it becomes $s+t^{-1}s+t^{-2}s+\dots+t^{-n}s+\cdots,$ so
under~\eqref{d.st}, these terms
form a {\em descending} chain; so that expression
does not describe a member of $k((G)).$
Rather, we find that $(1-t)^{-1}$ sends $s$ to the
element $-s t^{-1} -s t^{-2} -\dots -s t^{-n} -\cdots=
-ts -t^2s -\dots -t^ns -\cdots.$
(This is an example of the phenomenon noted in \S\ref{S.bij},
that given $x\in kG-\{0\},$ in this case $1-t,$
if we want to compute $g\,x^{-1}$ for some $g\in G,$
the member of $\supp(x)$
which behaves like the ``leading term'' of $x$ can depend on $g.)$
So in its action on $1\in k((G)),$ the operator
$(1-t)^{-1}$ ``looks like'' $\sum_{i\geq 0} t^i,$
while in its action on $s,$ it ``looks like'' $\sum_{i<0} -t^i.$

What if we ignore the ordering of $G$ that has allowed us to define
$k((G)),$ and simply calculate in the right $\!kG\!$-module $k^G$?
Then we find that right multiplication by $1-t$ takes
both $\sum_{i\geq 0} t^i$ and $\sum_{i<0} -t^i$ to $1.$
That these both can be true follows from the fact that $1-t$ annihilates
$(\sum_{i\geq 0} t^i) - (\sum_{i<0} -t^i)= \sum_{-\infty}^\infty t^i.$

This suggests that we look at a factor module of
the $\!kG\!$-module $k^G$ by a submodule
containing $\sum_{-\infty}^\infty t^i.$
Then the question ``$\sum_{i\geq 0} t^i$ or $\sum_{i<0} -t^i$?''
disappears -- these expressions represent the same element.
So perhaps, for a general right ordered group $G,$ we should,
in place of $k((G)),$ look at a module obtained
by dividing $k^G$ by some submodule of ``degenerate'' elements.
But it is not clear how to find such a factor module with
good properties; in particular,
how the right orderability of $G$ would be used.

\subsection{Partitioning $G\!$}\label{SS.partition_by_action}
Suppose $G$ is a right ordered group and $S$ a finite subset of $G.$
We have seen that for an element $g\in G,$ the ordering on
$gS\subseteq G$ need not be the one induced by the ordering
of $S\subseteq G;$ or to put it another way, the ordering $\leq_g$
on $S$ defined by $s\leq_g t \iff gs\leq gt$ can depend on $g.$
(We remark that since our ordering on $G$ is
right-invariant, the relation $gs\leq gt$ is
equivalent to $gsg^{-1}\leq gtg^{-1},$ i.e., to the
result of conjugating the given ordering on $G$ by $g.)$

However, since $S$ is finite, there are only finitely many
orderings on $S;$ so suppose we classify the elements $g\in G$
according to the restriction to $S$ of the ordering $\leq_g.$
If, for each total ordering $\preceq$ on $S,$ we define the subset
$G_\preceq=\{g\in G\mid ({\leq_g}|_S)=\,\preceq\},$
and write $k((G))$ as $\bigoplus_\preceq\,k((G_\preceq)),$
i.e., as the direct sum of the subspaces of elements
having supports in the various subsets $G_\preceq,$ then
we might expect multiplication by an element of $kG$ with
support in $S$ to be ``well-behaved'' on each summand of
this decomposition.

Unfortunately, some things we might hope for do not hold.
For instance, the subset $G_\preceq$ of $G$ containing $1,$ namely
the one for which $\preceq$ is the ordering of $S$ induced
by its inclusion in $G,$ need not be closed under multiplication.

To get an example of this, let us start with a free abelian group
on two generators $y$ and $z,$ and formally write $z$
as $y^\omega$ where $\omega$ denotes a primitive cube root
of unity, so that we can write the general element $y^i z^j$
of this groups as $y^{i+\omega j}$ $(i,j\in\Z).$
Now let $G$ be the extension of that abelian group by an element $x$
which acts by
\begin{equation}\begin{minipage}[c]{35pc}\label{d.y^*wh}
$y^h x\ =\ x\,y^{\omega h}$ \quad for $h\in\Z[\omega]$
\end{minipage}\end{equation}
(compare~\eqref{d.ts=}).
Since the group generated by the commuting
elements $y$ and $y^\omega$ is right orderable,
as is the group generated by $x,$ the same is true of the
extension group $G$ \cite[statement 3.7]{Conrad}.
(On the other hand, the fact that the group generated by $y$ and
$y^\omega$ has no ordering invariant under the action of $x$ implies
that $G$ has no two-sided invariant ordering.)
Let us fix a right-invariant ordering $\leq$ on $G.$
Note that the three elements $y,$ $y^\omega,$
$y^{\omega^2}$ of the orbit of $y$ under conjugation
by $\lang x\rang$ have product $1.$
From the fact that the positive cone of $\leq$ is closed
under multiplication, it follows that these three elements
are not all on the same side of $1$ with respect to $\leq;$ so two of
them must be on one side and the third on the other.
Thus, one of these three elements must have the
property that it stays on the same side of $1$ under
conjugation by $x,$ but moves to the opposite side
under conjugation by $x^2.$
Calling the member of $\{y,y^\omega,y^{\omega^2}\}$
which has this property $s,$ letting $S=\{1,s\},$
and letting $\preceq$ be the ordering on $S$ induced by
its inclusion in $G,$ we see that $1,x\in G_\preceq,$
but $x^2\notin G_\preceq.$

So it does not look easy to put this decomposition of $G$ to use.

Let us note a common generalization of the right-ordered
groups described by~\eqref{d.ts=} and~\eqref{d.y^*wh}.
Let $\Z[c,c^{-1}]$
be the subring of the complex numbers generated by a fixed nonzero
complex number $c$ and its inverse, and let us write the additive group
of $\Z[c,c^{-1}]$ multiplicatively as $y^{\Z[c,c^{-1}]}.$
Let $G$ be the extension of this group by the infinite cyclic
group generated by an elements $x,$ with the action
\begin{equation}\begin{minipage}[c]{35pc}\label{d.y^hx}
$y^h x\ =\ x\,y^{ch}$\quad for $h\in \Z[c,c^{-1}].$
\end{minipage}\end{equation}
We may order $G$ by letting
\begin{equation}\begin{minipage}[c]{35pc}\label{d.x^ny^h>}
$y^h x^n\ \geq\ y^{h'} x^{n'} \iff
\left \{ \begin{array}{cl}
\mbox{either} & n>n',\\[.2 em]
\mbox{or} & n=n'\ \ \mbox{and}\ \r{Re}(h)>\r{Re}(h'),\\[.2 em]
\mbox{or} & n=n',\ \, \r{Re}(h)=\r{Re}(h'),
\ \, \mbox{and}\ \, \r{Im}(h)\geq\r{Im}(h')\,.
\end{array} \right.$
\end{minipage}\end{equation}
(Cf.~\eqref{d.st}.)
In this situation, if $c$ has the form $e^{\alpha \pi i}$
$(\alpha\in\mathbb{R}),$ then for any $S$ with more than one element,
and any ordering $\preceq$ of $S$ such that $G_\preceq$ is
nonempty, it is not hard to show that $\{n\in\Z\mid x^n\in G_\preceq\}$
is periodic (invariant under some nonzero additive translation
on $\Z)$ if and only if $\alpha$ is rational.
So in the irrational case, the sets $G_\preceq$
are particularly messy.

\subsection{One case that would imply the general result we want}\label{SS.ord_aut_R}
Yves de~Cornulier (personal communication)
has pointed out that to prove embeddability
of $kG$ in a division ring for every right-orderable group $G,$
we `merely' need to prove this for $G$
the group of order-automorphisms of the ordered set of real numbers,
or, alternatively, for $G$ the order-automorphisms of the ordered
set of rationals.
For it is known \cite[Proposition~2.5]{Linnell} that any {\em countable}
right orderable group can be embedded in each of those groups; hence
if one of those two group algebras were embeddable in a division ring,
then for any right-orderable group $G,$ all of its finitely generated
subgroups $G_0$ would have group algebras $kG_0$ embeddable in
division rings, and
from this, a quick ultraproduct argument would give the embeddability
of $kG$ itself in a division ring.

\subsection{Can we use lattice-orderability?}\label{SS.lat_ord}
Recall the fact mentioned at the end of~\S\ref{S.intro},
that the one-sided-orderable groups are the groups embeddable,
group-theoretically, in lattice-ordered groups.
So what we want is equivalent to saying that
group algebra $kG$ of every lattice-ordered group $G$ is embeddable in
a division ring.
The partial ordering of a lattice-ordered group
is required to be invariant under both right and left translations,
and it is tempting to hope that we should be able to
construct a division ring of formal infinite sums whose
supports in $G$ have some nice property with respect to such
a lattice ordering.

However, note that any lattice-ordered group $G$ can be embedded
group-theoretically, by the diagonal map, in the lattice-ordered
group $G\times G^\r{\,op},$ where $G^\r{\,op}$
is the group $G$ with its partial order relation reversed.
Since the subgroup of $G\times G^\r{\,op}$ given by the
image of this embedding is an antichain, it is hard
to see how the order structure can be used to pick out a
class of infinite sums that would form a division ring and
contain that diagonal subring.

But one might be able to go somewhere with this idea -- perhaps
defining a permissible infinite sum not just in terms of
order relations among the elements of its support,
but using the sublattice generated by that support.
(Incidentally, the lattice structure of a lattice-ordered group is
always distributive \cite[Corollary~3.17]{Darnel}.)

\section{Appendix on prime matrix ideals}\label{S.PMC_versions}

Let us recall P.\,M.\,Cohn's approach to
maps of rings into division rings, which we sketched in \S\ref{S.intro}.
It is based on

\begin{definition}[\cite{FRR}, \cite{SF}, \cite{FRR+}]\label{D.PMC_sing_ker}
Let $f:R\to D$ be a homomorphism from a ring into a division ring.
Then the {\em singular kernel} $\Pm$ of $f$ is the
set of square matrices
over $R$ whose images under $f$ are singular matrices over $D.$
\end{definition}

Cohn shows that in the above situation, the structure of the division
subring of $D$ generated by $f(R)$ is determined by $\Pm$
(\cite{FRR}, \cite{SF}, \cite{FRR+}; see also \cite{Malcolmson1}),
and he notes that $\Pm$ has
properties~\eqref{d.PMC_nonfull}-\eqref{d.PMC_prime} below.

Let me explain in advance the notation
of~\eqref{d.PMC_nabla_col} and~\eqref{d.PMC_nabla_row}.
If $A$ and $B$ are square matrices of the same
size, which agree except in their
$\!r\!$-th row, or agree except in their
$\!r\!$-th column, then $A\nabla B$ is defined to be the matrix
which agrees with $A$ and $B$ in all rows or columns but the $\!r\!$-th,
and has for $\!r\!$-th row or column the sum of those
rows or columns of $A$ and $B.$
The specification of whether rows or columns are involved, and of
the $r$ in question, is understood to be determined by context.
Cohn calls $A\nabla B$
the {\em determinantal sum} of $A$ and $B,$ in view
of the expression, when $R$ is commutative,
for the determinant of that matrix.

Here, now, are the properties of the singular kernel $\Pm$ of
a homomorphism of $R$ into a division ring used by Cohn:
\begin{equation}\begin{minipage}[c]{35pc}\label{d.PMC_nonfull}
$\Pm$ contains every square $n\times n$ matrix that can be
written as the product of an $n\times n{-}1$ matrix and
an $n{-}1\times n$ matrix over $R.$
(Cohn calls such products {\em non-full} matrices.)
\end{minipage}\end{equation}
\begin{equation}\begin{minipage}[c]{35pc}\label{d.PMC_(+)}
If $A$ is a matrix lying in $\Pm,$ and $B$ is {\em any}
square matrix over $R,$ then $\Pm$ contains the matrix
$\left(\begin{matrix} A\ 0 \\
0\ B \end{matrix}\right),$ denoted $A\oplus B.$
\end{minipage}\end{equation}
\begin{equation}\begin{minipage}[c]{35pc}\label{d.PMC_nabla_col}
If $\Pm$ contains square $n\times n$ matrices $A$ and $B$ which agree
except in the $\!r\!$-th column for some $r,$ then it contains their
determinantal sum $A\nabla B$ with respect to that column.
\end{minipage}\end{equation}
\begin{equation}\begin{minipage}[c]{35pc}\label{d.PMC_nabla_row}
If $\Pm$ contains square $n\times n$ matrices $A$ and $B$ which agree
except in the $\!r\!$-th row for some $r,$ then it contains their
determinantal sum $A\nabla B$ with respect to that row.
\end{minipage}\end{equation}
\begin{equation}\begin{minipage}[c]{35pc}\label{d.PMC_(+)1}
If $\Pm$ contains a matrix of the form $A\oplus 1,$ where
$1$ denotes the $1\times 1$ matrix with entry $1,$ and where
$\oplus$ is defined as in~\eqref{d.PMC_(+)}, then $A\in\Pm.$
\end{minipage}\end{equation}
\begin{equation}\begin{minipage}[c]{35pc}\label{d.PMC_1}
The $1\times 1$ matrix $1$ is not in $\Pm.$
\end{minipage}\end{equation}
\begin{equation}\begin{minipage}[c]{35pc}\label{d.PMC_prime}
If $A\oplus B\in\Pm,$ then $A\in\Pm$ or $B\in\Pm.$
\end{minipage}\end{equation}

In~\cite{FRR},~\cite{FRR+}, and many other works,
Cohn calls a set of square matrices over a ring $R$ which
satisfies~\eqref{d.PMC_nonfull}-\eqref{d.PMC_(+)1} a
{\em matrix ideal}, and calls a matrix ideal which
also satisfies~\eqref{d.PMC_1} and~\eqref{d.PMC_prime} {\em prime}.
He proves that for every prime matrix ideal $\Pm$ of
$R,$ the ring gotten by universally adjoining to $R$ inverses
to all matrices not in $\Pm$ is a local ring, whose residue
ring is a division ring $D$ such that the singular kernel
of the induced map $R\to D$ is precisely $\Pm$
\cite[Theorem~7.4.3]{FRR+}.
Thus since, as mentioned, the singular kernel of a
map $f:R\to D$ determines the division subring generated by the
image of $R,$ it follows
that homomorphisms from $R$ into division rings generated
by the images of $R$ are, up to isomorphisms making
commuting triangles with those homomorphisms,
in bijective correspondence with prime matrix ideals of $R.$
We see from Definition~\ref{D.PMC_sing_ker}
that the homomorphism $R\to D$ corresponding to $\Pm$ is
one-to-one if and only if $\Pm$ contains no nonzero $1\times 1$ matrix.

However, in~\cite[\S4.4]{SF}, Cohn defines matrix ideals by
conditions~\eqref{d.PMC_nonfull}-\eqref{d.PMC_nabla_col}
and~\eqref{d.PMC_(+)1}, omitting~\eqref{d.PMC_nabla_row},
again calling such a matrix ideal prime
if~\eqref{d.PMC_1} and~\eqref{d.PMC_prime} hold.
(He notes at \cite[p.\,164, two lines after display~(30)]{SF}
that one can similarly define determinantal sums with respect
to rows, ``but this will not be needed''.)
He claims to prove, under this definition, the same result
cited above, that the prime matrix ideals
are precisely the singular kernels of homomorphisms to division rings.
This, together with the corresponding result proved
using the stronger definition, would imply that the two
definitions of prime matrix ideal are equivalent.

Now the shortened definition of prime matrix ideal
would lend itself to an approach similar to the one we took
in~\S\ref{S.M}.
Namely, given a right $\!R\!$-module $M,$ we could
for each $n\geq 0$ consider the $n\times n$ matrices over $R$
which act non-injectively on $M^n,$ verify that these together
satisfy {\em most} of the conditions to form a prime matrix ideal
(details below), and examine
when they satisfy the remaining conditions.
But this would be more difficult if we used the
definition appearing in most of Cohn's work on this
subject, containing condition~\eqref{d.PMC_nabla_row}.

Unfortunately, I have difficulty verifying
one of the steps in the proof in~\cite{SF} that prime
matrix ideals, defined without condition~\eqref{d.PMC_nabla_row},
yield homomorphisms to division rings.
Fortunately, Peter Malcolmson has been able to supply an
argument, which with his permission I give
below, showing that in the stronger definition of prime
matrix ideal, condition~\eqref{d.PMC_nabla_row} can be
replaced by a condition that {\em is} easily verifiable
for the set of matrices that act non-injectively on product
modules $M^n$ for a right $\!R\!$-module $M.$

Let me first sketch, for the reader with~\cite{SF} in
hand, my problem with the development given there.
It concerns the assertion in the middle of p.\,164 that the operation
$\odot$ on square matrices
introduced on that page respects equivalence classes under
the equivalence relation $\sim$ defined on p.\,163.
That equivalence relation is generated by three sorts
of operations on matrices:
certain operations of left multiplication by elementary matrices,
certain operations of right multiplication by elementary matrices,
and certain operations of deleting rows and columns.
If we have $a_1\sim a_2$ via a left multiplication operation,
or via the deletion operation, it is indeed straightforward
that $a_1\odot b\sim a_2\odot b$ via the same operation;
but if $a_1\sim a_2$ via a right multiplication operation,
I don't see why $a_1\odot b\sim a_2\odot b$ should hold.
Similarly, if $b_1\sim b_2$ via a right multiplication operation
or a deletion operation, I have no problem,
but if they are related via a left multiplication operation,
I don't see that $a\odot b_1\sim a\odot b_2.$

Here, however, is Malcolmson's result.

\begin{lemma}[P.\,Malcolmson, personal communication]\label{L.PM}
Let $R$ be a ring, and $\Pm$ a set of square matrices over $R$
satisfying~\eqref{d.PMC_nonfull}-\eqref{d.PMC_nabla_col}
and \eqref{d.PMC_(+)1}.
Then $\Pm$ also satisfies~\eqref{d.PMC_nabla_row} if and only
if it satisfies
\begin{equation}\begin{minipage}[c]{35pc}\label{d.I+eij}
For each $n>0,$ the set of $n\times n$ matrices in $\Pm$
is closed under left multiplication by matrices $I_n\pm e_{ij}$
$(i\neq j).$
\end{minipage}\end{equation}
\end{lemma}

\begin{proof}
``Only if'' follows from \cite[2nd ed., point (f) on p.\,398]{FRR},
which shows that a set $\Pm$ of square matrices
satisfying conditions~\eqref{d.PMC_nonfull}-\eqref{d.PMC_(+)1}
(there called M.1-M.4, with M.3 being the conjunction
of~\eqref{d.PMC_nabla_col} and~\eqref{d.PMC_nabla_row})
is closed under right and left multiplication by
arbitrary square matrices.
Below, we shall prove ``if\,''; so assume~\eqref{d.I+eij} holds.

By a familiar calculation, the group generated by the
elementary matrices $I+e_{ij}$ and their inverses
$I-e_{ij}$ contains the matrices whose
left actions transpose an arbitrary pair of rows, changing
the sign of one of them.
(The essence of that calculation is the $2\times 2$ case,
$\left(\begin{matrix} 1 & 0 \\
1 & 1 \end{matrix}\right)
\left(\begin{matrix} 1 & \!{-}1 \\
0 & 1 \end{matrix}\right)
\left(\begin{matrix} 1 & 0 \\
1 & 1 \end{matrix}\right)=
\left(\begin{matrix} 0 & \!{-}1 \\
1 & 0 \end{matrix}\right).)$
This will be a key tool later on, but let us first
use it in a trivial way: it allows us to reduce to the case where
the row with respect to which we want to show closure
under determinantal sums is the last row of our matrices.
(That reduction also uses the observation that the operation of
determinantal sum with respect to any row respects the operation of
reversing the sign of a particular row in all matrices.)

Another fact we shall use is that if $\Pm$ is a set of square
matrices satisfying~\eqref{d.PMC_nonfull} and~\eqref{d.PMC_nabla_col},
$A$ an $n\times n$ matrix,
$B$ an $n'\times n'$ matrix, and
$C$ an $n'\times n$ matrix, then
\begin{equation}\begin{minipage}[c]{35pc}\label{d.low_left}
$\Pm$ contains
$\left(\begin{matrix} A & 0 \\
0 & B \end{matrix}\right)$
if and only if it contains
$\left(\begin{matrix} A & 0 \\
C & B \end{matrix}\right).$
\end{minipage}\end{equation}
This can be seen from point~(e) on p.\,397
of \cite[2nd edition]{FRR}.
(Although both~\eqref{d.PMC_nabla_col} and~\eqref{d.PMC_nabla_row} are
assumed there, only the former is used in the calculation.)

Now to prove the ``if'' direction of our lemma, let $H$ be an
$n{-}1\times n$ matrix over $R,$ and $a,$ $b$ length-$\!n\!$
rows such that
\begin{equation}\begin{minipage}[c]{35pc}\label{d.XaXb}
$\left(\begin{matrix} X \\
a \end{matrix}\right),
\ \left(\begin{matrix} X \\
b \end{matrix}\right)\in\Pm.$
\end{minipage}\end{equation}
Applying~\eqref{d.PMC_(+)}, we get
$\left(\begin{matrix} X & 0 \\
a & 0 \\
0 & 1\end{matrix}\right),
\ \left(\begin{matrix} X & 0 \\
b & 0 \\
0 & 1\end{matrix}\right)\in\Pm.$
Applying~\eqref{d.low_left} to these matrices, we get
$\left(\begin{matrix} X & 0 \\
a & 0 \\
b & 1\end{matrix}\right),
\ \left(\begin{matrix} X & 0 \\
b & 0 \\
{-}a{-}b & 1\end{matrix}\right)\in\Pm.$
If we left-multiply the first of those
two matrices by $I_n+e_{n-1,n},$ we get
$\left(\begin{matrix} X & 0 \\
a{+}b & 1 \\
b & 1\end{matrix}\right)\in\Pm,$ while if we left multiply the
second by a product of elementary matrices that
transposes the last two rows and changes the sign of one of
them, we get $\left(\begin{matrix} X & 0 \\
a{+}b & -1 \\
b & 0 \end{matrix}\right)\in\Pm.$

These two matrices differ only in their last column, and
applying~\eqref{d.PMC_nabla_col} to their
determinantal sum with respect to that column
gives $\left(\begin{matrix} X & 0 \\
a{+}b & 0 \\
b & 1 \end{matrix}\right)\in\Pm.$
Applying~\eqref{d.low_left} again, this gives
$\left(\begin{matrix} X & 0 \\
a{+}b & 0 \\
0 & 1 \end{matrix}\right)\in\Pm,$
hence by~\eqref{d.PMC_(+)1},
$\left(\begin{matrix} X \\
a{+}b \end{matrix}\right)\in\Pm.$
Having gotten this from~\eqref{d.XaXb}, we have
proved the case of~\eqref{d.PMC_nabla_row} where $r=n,$
which we have seen is equivalent to the general case.
\end{proof}

We can now obtain a result parallel to Lemma~\ref{L.cl_fr_M}.
As in the context of that lemma, elements of $M^n$ will be regarded as
row vectors, on which $n\times n'$ matrices over $R$ act on the right.
(Thus, the kernel $K$ referred to in~\eqref{d.PMC_nabla_col_if} below
is not, in general, an
$\!R\!$-submodule of $M^n,$ but merely an additive subgroup.)

\begin{lemma}\label{L.non-f+row-sum}
Let $M$ be a nonzero right module over a ring $R,$ and $\Pm$ the set of
square matrices $A$ over $R$ such that, if $A$ is $n\times n,$
$A$ gives a non-injective map $M^n\to M^n.$
Then\vspace{.5em}

\textup{(i)} $\Pm$ satisfies~\eqref{d.PMC_(+)}, \eqref{d.PMC_(+)1},
\eqref{d.PMC_1}, \eqref{d.PMC_prime}, and~\eqref{d.I+eij}.\vspace{.5em}

\textup{(ii)} A necessary and sufficient condition for
$\Pm$ to satisfy~\eqref{d.PMC_nonfull} is
\begin{equation}\begin{minipage}[c]{35pc}\label{d.PMC_nonfull_iff}
No $n\times n{-}1$ matrix over $R$ induces an injection of
abelian groups $M^n\to M^{n-1}$ $(n>0).$
\end{minipage}\end{equation}
\vspace{-1em} 

\textup{(iii)}
A sufficient condition for $\Pm$ to satisfy~\eqref{d.PMC_nabla_col}
is
\begin{equation}\begin{minipage}[c]{35pc}\label{d.PMC_nabla_col_if}
If $K\subseteq M^n$ is the kernel of the
action on $M^n$ of an $n\times n{-}1$ matrix over $R,$ then
either \textup{(a)}~every
map $M^n\to M$ which is induced by a height-$\!n\!$ column vector
over $R,$ and is nonzero on $K,$ is one-to-one on $K,$
or \textup{(b)}~no such map is one-to-one on $K.$
\end{minipage}\end{equation}

Thus, if both~\eqref{d.PMC_nonfull_iff} and~\eqref{d.PMC_nabla_col_if}
hold, then $\Pm$ is a prime matrix ideal of $R.$
Hence if, further, the right $\!R\!$-module $M$ is faithful, then $R$
is embeddable in a division ring.
\end{lemma}

\begin{proof}
All parts of~(i) are straightforward.
(Condition \eqref{d.I+eij} is a special case of the observation that
$\Pm$ is closed under left and right multiplication by arbitrary
invertible matrices.)

(ii) is also easy:
Assume first that $\Pm$ satisfies~\eqref{d.PMC_nonfull}.
If $A$ is an $n\times n{-}1$ matrix over $R,$ then extending $A$
by a zero column, we get an $n\times n$ matrix $A'$
which is non-full in the sense stated in~\eqref{d.PMC_nonfull},
hence by~\eqref{d.PMC_nonfull} lies in $\Pm,$ hence, by our choice of
$\Pm,$ is not one-to-one on $M^n.$
Hence $A$ is not one-to-one there, proving~\eqref{d.PMC_nonfull_iff}.
Conversely, if $A$ is a non-full $n\times n$ matrix,
say $A=BC$ where $B$ is $n\times n{-}1$ and $C$ is $n{-}1\times n,$
then assuming~\eqref{d.PMC_nonfull_iff}, $B$ acts on $M^n$ with
nonzero kernel, hence so does $A,$ so $A\in\Pm.$

To prove (iii), let $A,B\in\Pm$ be as in~\eqref{d.PMC_nabla_col},
$C$ the common $n\times n{-}1$ submatrix obtained by deleting
the $\!r\!$-th columns from these, and $K$
the kernel of the action of $C$ on $M^n.$
From the fact that $A,B\in\Pm,$ we see that $K\neq\{0\}.$
Now if, as in the first alternative of~\eqref{d.PMC_nabla_col_if},
every map $M^n\to M$ induced by a height-$\!n\!$ column
vector restricts to either the zero map or a one-to-one map on $K,$
then for $A$ and $B$ to lie in $\Pm,$
their $\!r\!$-th columns must both induce the zero map
on $K,$ hence so will the sum of those columns,
showing (since $K\neq\{0\})$ that $A\nabla B$ lies in $\Pm.$
On the other hand, if {\em no} height-$\!n\!$ column vector induces a
one-to-one map on $K,$ then in particular, the
$\!r\!$-th column of $A\nabla B$ does not, giving the same conclusion.

To see the first sentence of the last paragraph of the lemma,
note that~(i),~(ii) and~(iii) give us all of
\eqref{d.PMC_nonfull}-\eqref{d.PMC_prime}
except~\eqref{d.PMC_nabla_row}, and that is given to us by
Lemma~\ref{L.PM}, since~(i) includes~\eqref{d.I+eij}.
The final sentence follows by the results of \cite{FRR} cited earlier.
\end{proof}

Remark: The converse of~(iii) above is not true;
i.e., $\Pm$ can satisfy~\eqref{d.PMC_nabla_col}
without satisfying~\eqref{d.PMC_nabla_col_if}.
For example, suppose $R=\Z$ and $M$ is the module $\Z/p^2\Z$
for some prime $p.$
It is not hard to see that the $\Pm$ of Lemma~\ref{L.non-f+row-sum}
will consist of the square matrices over $\Z$ whose
determinants are divisible by $p.$
This is the prime
matrix ideal corresponding to the homomorphism
of $\Z$ into the field $\Z/p\Z,$ so
in particular, it satisfies~\eqref{d.PMC_nabla_col}.
On the other hand, for any $n\geq 1,$ the subgroup
$K=\{0\}^{n-1}\times M\subseteq M^n$ is easily seen to be
the kernel of the action
of an $n\times n{-}1$ matrix; but if we take a height-$\!n\!$
column vector with $1$ in the $\!n\!$-th position, and
another with $p$ in that position, then both
are nonzero on $K,$ but the former is one-to-one while the
latter is not; so~\eqref{d.PMC_nabla_col_if} fails.\vspace{.4em}

On a general note,
the above approach to obtaining homomorphisms into division
rings from modules may be thought of as less convenient than the one
developed in \S\ref{S.M}, in that it leaves us the two
conditions~\eqref{d.PMC_nonfull_iff} and~\eqref{d.PMC_nabla_col_if}
to verify, in contrast to the one condition~\eqref{d.cl_exch_iff}
(with equivalent
forms~\eqref{d.cl_exch_iff_mx-},~\eqref{d.cl_exch_iff_mx}).
But it is, in another way, more robust,
in that the concept of prime matrix ideal is left-right symmetric,
and this allows us to produce a version of the same result based on
surjectivity rather than injectivity {\em without} switching
from right to left modules as we did in that section.
Rather, the switch between injectivity and
surjectivity can be made independently of whether we
use right or left modules.
The next lemma is the result based on right modules and surjectivity;
the two left-module results are obtained from the two right-module
results in the obvious way.
We leave to the reader the proof of the lemma, which
exactly parallels that of Lemma~\ref{L.non-f+row-sum}.

\begin{lemma}\label{L.non-f+col-sum}
Let $M$ be a nonzero right module over a ring $R,$ and $\Pm$ the set of
square matrices $A$ over $R$ such that, if $A$ is $n\times n,$
$A$ gives a {\em non-surjective} map $M^n\to M^n.$
Then\vspace{.5em}

\textup{(i)} $\Pm$ satisfies~\eqref{d.PMC_(+)}, \eqref{d.PMC_(+)1},
\eqref{d.PMC_1}, \eqref{d.PMC_prime}, and
the left-right dual of~\eqref{d.I+eij} \textup{(}closure under
{\em right} multiplication by matrices $I_n\pm e_{ij}).$\vspace{.5em}

\textup{(ii)} A necessary and sufficient condition for
$\Pm$ to satisfy~\eqref{d.PMC_nonfull} is
\begin{equation}\begin{minipage}[c]{35pc}\label{d.PMC_nonfull_iff'}
No $n{-}1\times n$ matrix over $R$ induces a surjection of
abelian groups $M^{n-1}\to M^n$ $(n>0).$
\end{minipage}\end{equation}
\vspace{-1em}

\textup{(iii)}
A sufficient condition for $\Pm$ to satisfy~\eqref{d.PMC_nabla_row}
is
\begin{equation}\begin{minipage}[c]{35pc}\label{d.PMC_nabla_row_if}
If $I\subseteq M^n$ is the image of $M^{n-1}$ under the
action of an $n{-}1\times n$ matrix over $R,$ then
either \textup{(a)}~every
map $M\to M^n$ which is determined by a length-$\!n\!$ row vector
and has image not contained in $I$ has image which, with $I,$
spans the additive group of $M^n,$ or \textup{(b)}~no
such map has image which, with $I,$ spans that additive group.
\end{minipage}\end{equation}

Thus, if both~\eqref{d.PMC_nonfull_iff'} and~\eqref{d.PMC_nabla_row_if}
hold, then $\Pm$ is a prime matrix ideal of $R.$
Hence if, further, every nonzero element of $R$ carries $M$ surjectively
to itself, then $R$ is embeddable in a division ring.\qed
\end{lemma}

\section{Acknowledgements}\label{S.Ackn}
I am indebted to Pace Nielsen for many invaluable comments and
corrections to earlier drafts of this note,
to A.\,M.\,W.\,Glass
for a helpful discussion of what is known about right-orderable groups,
to Peter Malcolmson for showing the way to Lemma~\ref{L.PM},
and to Yves de~Cornulier for the observation of~\S\ref{SS.ord_aut_R}.

\Needspace{4\baselineskip}

\end{document}